\documentclass[10pt]{amsart}

\usepackage{amsmath, amscd, amssymb, euler}
\usepackage[frame,cmtip,arrow,matrix,line,graph,curve]{xy}
\usepackage{graphpap, color,bbm, verbatim}
\usepackage[mathscr]{eucal}
\usepackage[normalem]{ulem}

\numberwithin{equation}{section}

\newcommand{\CC}{\mathbb{C}}

\newcommand{\PP}{\mathbb{P}}
\newcommand{\QQ}{\mathbb{Q}}

\newcommand{\ZZ}{\mathbb{Z}}

\newcommand{\bB}{\mathbf{B}}

\newcommand{\bS}{\mathbf{S}}
\newcommand{\bM}{\mathbf{M}}
\newcommand{\bH}{\mathbf{H}}

\newcommand{\bR}{\mathbf{R}}

\newcommand{\cal}{\mathcal}

\def\cA{{\cal A}}

\def\cC{{\cal C}}

\def\cE{{\cal E}}
\def\cF{{\cal F}}

\def\cM{{\cal M}}

\def\cO{{\cal O}}

\def\cS{{\cal S}}

\def\cU{{\cal U}}

\def\cW{{\cal W}}

\def\mapright#1{\,\smash{\mathop{\lra}\limits^{#1}}\,}

\def\lra{\longrightarrow}

\newcommand{\ses}[3]{0\lr{#1}\lr{#2}\lr{#3}\lr 0}

\def\and{\quad{\rm and}\quad}

\def\and{\quad\text{and}\quad}
\def\mapright#1{\,\smash{\mathop{\lra}\limits^{#1}}\,}

\def\PP{\mathbb{P}}
\def\CC{\mathbb{C}}

\def\lra{\longrightarrow}
\def\mapright#1{\,\smash{\mathop{\lra}\limits^{#1}}\,}
\def\cO{\mathcal{O}}

\def\git{/\!/ }

\def\lr{\longrightarrow}

\input xy
\xyoption{all}

\DeclareMathOperator{\Ext}{Ext} 
\DeclareMathOperator{\Hom}{Hom}

 \DeclareMathOperator{\spec}{Spec}
\DeclareMathOperator{\Sym}{Sym} 

\newtheorem{prop}{Proposition}[section]
\newtheorem{theo}[prop]{Theorem}
\newtheorem{lemm}[prop]{Lemma}
\newtheorem{coro}[prop]{Corollary}
\newtheorem{rema}[prop]{Remark}
\newtheorem{defi}[prop]{Definition}

\newtheorem{exam}[prop]{Example}

\newtheorem{defiprop}[prop]{Definition-Proposition}

\title[Compactified moduli of curves]{Compactified moduli
spaces of rational curves in projective homogeneous varieties}

\author{Kiryong Chung}
\address{Department of Mathematics, Seoul National University, Seoul 151-747, Korea}
\email{dragon10@snu.ac.kr}

\author{Jaehyun Hong}
\address{Department of Mathematics and Research Institute
of Mathematics, Seoul National University, Seoul 151-747, Korea}
\email{jhhong00@snu.ac.kr}

\author{Young-Hoon Kiem}
\address{Department of Mathematics and Research Institute
of Mathematics, Seoul National University, Seoul 151-747, Korea}
\email{kiem@math.snu.ac.kr}
\thanks{Partially supported by NRF.}

\begin{document}
\begin{abstract}
The space of smooth rational curves of degree $d$ in a projective variety $X$ has compactifications by taking closures in the Hilbert scheme, the moduli space of stable sheaves or the moduli space of stable maps respectively. In this paper we compare these compactifications by explicit blow-ups and -downs when $X$ is a projective homogeneous variety and $d\leq 3$. Using the comparison result, we calculate the Betti numbers of the compactifications when $X$ is a Grassmannian variety.
\end{abstract}
\maketitle
\section{Introduction}
Let $X$ be a smooth projective variety over $\CC$
with fixed embedding $i:X \hookrightarrow \PP^r$.
Let $\bR (X,d)$ be the moduli space of all smooth rational curves of
degree $d$ in $X$.
It is well known that $\bR (X,d)$ is smooth when $X$ is a convex variety in the sense that $H^1(\PP^1,f^*T_X)=0$ for any morphism $f:\PP^1\to
X$ of degree $d$.
However, even for projective spaces, when $d\geq 2$, $\bR (X,d)$ is not compact.
From moduli theoretic point of view, the following questions are quite natural:
\begin{enumerate}
\item Does $\bR (X,d)$ admit a moduli theoretic compactification?
\item If there are more than one such compactifications,
what are the relationships among them?
\item Can we calculate topological invariants of the compactifications,
such as the Betti numbers and intersection numbers?\end{enumerate}

As we will see below, there are several well-known compactifications
of $\bR (X,d)$ by Hilbert scheme, the moduli space of semistable sheaves
and the moduli space of stable maps. The purpose of this paper is to provide
answers to the second and third questions when the target $X$ is a homogeneous
projective variety and $d$ is at most 3.

Let us recall several important compactifications of $\bR (X,d)$.

\medskip
\noindent$\bullet$ \textbf{Hilbert compactificaiton}: Since $X\subset \PP^r$ is a projective variety, Grothendieck's general
construction gives us the Hilbert scheme $Hilb^{dm+1}(X)$ of
closed subschemes of $X$ with Hilbert polynomial $h(m)=dm+1$ as a
closed subscheme of $Hilb^{dm+1}(\PP^r)$. The inclusion
$\bR (X,d)\subset Hilb^{dm+1}(X)$ is an open immersion and thus the
irreducible component(s) of $Hilb^{dm+1}(X)$ containing smooth
rational curves is a compactification which we call the
\emph{Hilbert compactificaiton} and denote by $\bH (X,d)$.

\medskip
\noindent $\bullet$ \textbf{Kontsevich
compactification}: In 1994, Kontsevich and Manin proposed another way to
compactify $\bR (X,d)$ by using the notion of stable maps. A
stable map is a morphism of a connected nodal curve $f:C\to X$
with finite automorphism group. Recall that two maps $f:C\to X$
and $f':C'\to X$ are isomorphic if there exists an isomorphism
$\eta:C\to C'$ satisfying $f'\circ \eta=f$. Let $\cM_0(X,d)$
denote the (coarse) moduli space of isomorphism classes of stable
maps $f:C\to X$ with arithmetic genus of $C$ equal to $0$ and
$\mathrm{deg}(f^*\cO_X(1))=d$. The obvious inclusion
$\bR (X,d)\to \cM_0(X,d)$ is an open immersion and hence the
closure $\bM (X,d)$ of $\bR (X,d)$ in $\cM_0(X,d)$ is a
compactification, which we call the \emph{Kontsevich
compactification}.

\medskip
\noindent $\bullet$ \textbf{Simpson compactification}: Yet another natural compactification is obtained by using C.
Simpson's general construction of moduli spaces of semistable sheaves
on a projective variety $X\subset \PP^r$. A coherent sheaf $E$ on
$X$ is \emph{pure} if any nonzero subsheaf of $E$ has the same
dimensional support as $E$. A pure sheaf $E$ is called
\emph{semistable} if \[ \frac{\chi(E(m))}{r(E)}\le
\frac{\chi(E''(m))}{r(E'')}\qquad \text{for }m>\!>0
\]
for any nontrivial pure quotient sheaf $E''$ of the same
dimension, where $r(E)$ denotes the leading coefficient of the
Hilbert polynomial $\chi(E(m))=\chi(E\otimes \cO_X(m))$. We obtain
\emph{stability} if $\le $ is replaced by $<$. If we replace the
quotient sheaves $E''$ by subsheaves $E'$ and reverse the
inequality, we obtain an equivalent definition of (semi)stability.

Simpson proved that there is a projective moduli scheme $\cS
imp^{P}(X)$ of semistable sheaves of given Hilbert polynomial $P$.
If $C$ is a smooth rational curve in $X$, then the structure sheaf
$\cO_C$ is a stable sheaf on $X$. Hence we get an open immersion
$\bR (X,d)\hookrightarrow \cS imp^{dm+1}(X)$. By taking the
closure we obtain a compactifiction $\bS (X,d)$, which we call the
\emph{Simpson compactification}.

\medskip
\begin{rema}\label{rem1.1}
When $X\neq \PP^r$, $\bM (X,d)$ may have many components:
The embedding $i:X \hookrightarrow \PP^r$ induces a homomorphism
$ i_*: H_2(X,\ZZ) \rightarrow H_2(\PP^r,\ZZ)\cong \ZZ$ and
$\bM (X,d)$ is decomposed as
\[\coprod_{\beta\in H_2(X,\ZZ),i_*(\beta)=d }\bM (X,\beta).\]
If $X$ is a projective homogeneous variety, then each $\bM (X,\beta)$ is an irreducible variety \cite{Kim-Pand}.
Similarly, $\bS (X,d)$ and $\bH (X,d)$ may be disjoint unions of components $\bS (X,\beta)$
and $\bH (X,\beta)$ respectively. Note that, by definition, $\bM(X,\beta)$, $\bS(X,\beta)$ and $\bH(X,\beta)$ are birational and thus they are all irreducible.
\end{rema}

We will often write $\bM$ or $\bM(X)$ (resp. $\bS$ or $\bS(X)$,
resp. $\bH$ or $\bH(X)$) instead of $\bM (X,d)$ (resp. $\bS (X,d)$,
resp. $\bH (X,d)$) when the meaning is clear from the context.
Now we can phrase more precisely the problem of interest in this paper as follows.

\bigskip

\noindent \textbf{Problem:} Compare the compactifications $\bH$, $\bM$ and $\bS$
explicitly and calculate the Betti numbers of them.

\bigskip
In \cite{Kiem,CK}, the authors solved this problem for $X=\PP^r$ and $d =2, 3$. When $d=1$, all the compactifications coincide with the Grassmannian $Gr(2,r+1)$. When $d=2$, we proved the following.

\begin{theo}\label{th1.1}\cite[\S4]{Kiem} Let $r\ge 2$. \begin{enumerate}
\item $\bS (\PP^r,2)\cong \bH (\PP^r,2)$.
\item The blow-up of $\bM(\PP^r,2)$ along
the locus of stable maps with linear image
coincides with the smooth blow-up of $\bS (\PP^r,2)$ along the locus of stable sheaves with linear support.
\end{enumerate}
\end{theo}
The isomorphism (1) follows directly from the fact that
the structure sheaf of every conic in $\PP^r$ is a stable sheaf.
To prove (2), we first showed that $\bM (\PP^r,2)$ is in fact Kirwan's partial
desingularization of the GIT quotient $$\PP(\Sym^2(\CC^2)\otimes \CC^{r+1})\git SL(2)$$
where $SL(2)$ acts on $\Sym^2(\CC^2)$ in the standard fashion and trivially on $\CC^{r+1}$.
As a consequence $\bM (\PP^r,2)$ is an $SL(2)$-quotient of a smooth variety $P_1$,
which is the stable part of a smooth blow-up of the semistable part $P_0^{ss}$ of the projective space $P_0=\PP(\Sym^2(\CC^2)\otimes \CC^{r+1})$.
There is a family of stable maps parameterized by the stable part $P_1^s$.
We blow up $P_1^s$ along the locus of stable maps $f:C\to \PP^r$ with linear image and apply elementary modification to transform the direct image sheaves $f_*\cO_C$ into stable sheaves.
This gives us a morphism to $\bS (\PP^r,2)$.
By analyzing the normal bundle of the exceptional locus we could prove
that the induced morphism is in fact the blow-up above.

For $X=\PP^r$ ($r\ge 3$) and $d=3$, we proved in \cite{CK} that
$\bM(\PP^r,3),\bS(\PP^r,3)$ and $\bH(\PP^r,3)$
are related by explicit (weighted) blow-ups as follows.
\begin{theo}\label{th1.2} Let $r\ge 3$.
\begin{enumerate}
 \item $\bH(\PP^r,3)$ is the smooth blow-up of
  $\bS(\PP^r,3)$ along the locus $\Delta(\PP^r)$
of planar stable sheaves.
  \item  $\bS(\PP^r,3)$
  is obtained from $\bM(\PP^r,3)$ by three weighted blow-ups
  followed by three weighted blow-downs.
  (See \S\ref{sec1.1} for a more precise description.)
\end{enumerate}
\end{theo}
As an application of this theorem, we could calculate
all the Betti numbers of the compactificatoins.

The purpose of this paper is to generalize the above theorems for all projective homogeneous varieties.
The main theorem of this paper may be phrased as follows.
\begin{theo}\label{main-thm}
Theorems \ref{th1.1} and \ref{th1.2} hold for any projective homogeneous variety $X$ in $\PP^r$.
\end{theo}
\begin{rema}
As we considered in Remark \ref{rem1.1}, when $H_2(X,\ZZ)\neq \ZZ$,
$\bM (X,d)$ may have disjoint irreducible components $\bM (X,\beta)$ and
so do $\bS (X,d)$, $\bH (X,d)$.
Hence Theorem \ref{main-thm} is really about the birational geometry of $\bM (X,\beta) , \bS (X,\beta)$ and $\bH (X,\beta)$.
\end{rema}
\begin{rema}
For the proof of Theorem \ref{main-thm}, we only use four properties of a projective homogeneous variety $X$
which will be stated in Lemma \ref{assu5-1}.
\end{rema}

As a direct application, we can calculate
all the Betti numbers of $\bH(X,d)$ and $\bS(X,d)$ for Grassmannians $X=Gr(k,n)$ and $d\le 3$.
See Corollaries \ref{thm41} and \ref{thm42} for precise closed formulas.
The Betti numbers of $\bM (Gr(k,n),d)$ for $d\le 3$ have been calculated by A. L\'opez-Mart\'in in \cite{Martin}.


\subsection{Outline of this paper} \label{sec1.1}
In \S\ref{sec2} and \S\ref{sec3.1},
we compare $\bM(X,d)$ and $\bS(X,d)$ for $d=2,3$ respectively.
We first define a rational map $$\bar{\phi}:\bM(X,d) \dashrightarrow \bS(X,d),\qquad f \mapsto f_* \cO_C$$ where $f:C\lr X$ is a stable map.
Then the undefined locus $\Gamma(X,d)$ of the rational map
$\bar{\phi}$ is the locus of stable maps with multiple components, i.e.
there exists a component $C_1$ of $C$ such that $f|_{C_1}$ is not
generically one-to-one.

When $d=2$, $\Gamma(X,2)$ is isomorphic to $\bM(\PP \cU,2)$
where $\cU$ is the tautological rank 2 bundle over the moduli
space $F_1(X)=\cM_0(X,1)$ of lines in $X$.
By Remark \ref{rem1.1}, $\Gamma(X,2)$ is in fact a disjoint union of $\Gamma(X,\beta)\subset\bM(X,\beta)$ where $\Gamma(X,\beta)$ is a $\bM(\PP^1 ,2)$-bundle
over the moduli space $F_1(X, \frac{\beta}{2} )$ of lines in $X$ with homology class $\frac{\beta}{2} \in H_2(X,\ZZ)$
such that $i_*\frac{\beta}{2} =[\PP^1]\in H_2(\PP^r,\ZZ)$.

If we blow up $\bM(X,2)$ along $\Gamma(X,2)$ and apply the elementary modification (\cite[Chapter 5]{HL}) along the exceptional divisor with respect to the first terms in the Harder-Narasimhan filtrations, we obtain a family of stable sheaves and thus a birational morphism from the blown-up space to $\bS(X,2)$.
Then we analyze a neighborhood of the exceptional divisor
and check that this morphism is in fact a blow-up map along the locus of sheaves with linear support.

When $d=3$, we can apply the same line of ideas
but things are more complicated. By taking the direct image $f_*\cO_C$ for each $f:C\to X$ in $\bM=\bM (X,3)$, we have a family of coherent sheaves $\cE_0$ on $\bM\times X$, flat over $\bM$, and a birational map $\bar{\phi}:\bM\dashrightarrow \bS=\bS(X,3)$.
The locus of unstable sheaves is the union of two subvarieties;\begin{enumerate}
\item the locus $\Gamma_0^1$ of stable maps whose images are lines,
\item the locus $\Gamma_0^2$ of stable maps whose images consist of two lines.\end{enumerate}
As in the case of the degree $2$, the unstable loci $\Gamma_0^i$ are in fact the disjoint union of irreducible components with respect to the homology class
and thus their dimensions may vary from components to components of $\bM(X,3)$.
For a $f\in \Gamma_0^1$ whose image is a line $L$, $f_*\cO_C=\cO_L\oplus \cO_L(-1)^2$ and the normal space of $\Gamma^1_0$ in $\bM$ at $f$ is
\[
\Hom (\CC^2,\Ext^1_X(\cO_L,\cO_L(-1))).
\]
Let $\pi_1:\bM_1\to \bM$ denote the blow-up along $\Gamma_0^1$. The destabilizing quotients $f_*\cO_L\to \cO_L(-1)^2$ form a flat family $A$ over the exceptional divisor $\Gamma^1_1$ of $\pi$ and
by applying the elementary modification with respect to this family of quotients, we obtain a family $\cE_1$ of coherent sheaves on $X$ parameterized by $\bM_1$. By direct calculation, we find that the locus of unstable sheaves in $\bM_1$ still consists of two subvarieties; \begin{enumerate}
\item the proper transform $\Gamma_1^2$ of $\Gamma_0^2$,
\item the subvariety $\Gamma_1^3$ of the exceptional divisor $\Gamma_1^1$ which are fiber bundles over $\Gamma^1_0$ with fibers
$$\PP \Hom_1 (\CC^2,\Ext^1_X(\cO_L,\cO_L(-1)))\cong \PP^1\times \PP \Ext^1_X(\cO_L,\cO_L(-1))$$
where $\Hom_1$ denotes the locus of rank 1 homomorphisms.
\end{enumerate}

Next we take the blow-up $\bM_2$ of $\bM_1$ along $\Gamma_1^2$
and apply elementary modification along the exceptional divisor $\Gamma_2^2$.
Then we find that the locus of unstable sheaves is precisely $\Gamma^3_2$
where $\Gamma_2^j$ is the proper transform of $\Gamma_1^j$ for $j=1,3$.
We let $\pi_3:\bM_3\to\bM_2$ be the blow-up of $\bM_2$ along $\Gamma_2^3$
and apply elementary modification along the exceptional divisor $\Gamma_3^3$.
We let $\Gamma^j_3$ denote the proper transform of $\Gamma^j_2$ for $j=1,2$.
The upshot is a family of stable sheaves on $X$ parameterized by $\bM_3$
and thus a morphism $\bM_3\to \bS$.

To analyze the morphism $\bM_3\to \bS$, we keep track of analytic
neighborhoods of $\Gamma^1_0$ and $\Gamma^2_0$ through the sequence of
blow-ups (and -downs). It turns out that the local geometry is
completely determined by variation of GIT quotients. For instance,
a neighborhood of $\Gamma_1^1$ is a fiber bundle over $\Gamma_0^1$ with fibers the GIT
quotient of $\cO_{\PP^7\times \PP^{2t-1}}(-1,-1)$ by $SL(2)$ with
respect to the linearization $\cO(1,\lambda)$ for $0<\lambda<\!<1$
where $t=\dim \Ext^1_X(\cO_L,\cO_L(-1))$.
As we vary $\lambda$ from $0^+$ to $\infty$, the GIT quotient goes
through two flips, or two blow-ups followed by two blow-downs. The
two blow-ups correspond to our two blow-ups $\bM_3\to\bM_2\to
\bM_1$ and we can blow down twice $\bM_3\to \bM_4\to \bM_5$ in the
neighborhoods of $\Gamma^1_j$. For $\lambda>\!>1$, the GIT quotient
of $\PP^7\times \PP^{2t-1}$ by $SL(2)$ is a $\PP^7$-bundle which
can be contracted in the open neighborhood. A similar analysis for
a neighborhood of $\Gamma^2$ tells us that we can blow down
$\bM_3$ three times
\[\bM_3\to \bM_4\to \bM_5\to \bM_6\]
and the morphism $\bM_3\to \bS$ is constant on the fibers of the
blow-downs. Hence we obtain an induced morphism $\bM_6\to \bS$
which turns out to be injective. So we conclude that $\bM_6\cong
\bS$.

We can summarize the above discussion as follows.
\begin{theo}\label{th1.3} For a projective homogeneous variety $X\subset \PP^r$,
$\bS=\bS(X,3)$ is obtained from $\bM=\bM(X,3)$ by blowing up along $\Gamma^1_0$,
$\Gamma^2_1$, $\Gamma^3_2$ and then blowing down along
$\Gamma^2_3$, $\Gamma^3_4$, $\Gamma^1_5$ where $\Gamma^j_i$ is the
proper transform of $\Gamma^j_{i-1}$ if $\Gamma^j_{i-1}$ is not
the blow-up/-down center and the image/preimage of
$\Gamma^j_{i-1}$ otherwise.
\end{theo}

Next we compare $\bH(X,d)$ and $\bS(X,d)$.
By Theorem \ref{th1.1} (1), when $d=2$, the Hilbert compactification $\bH(X,2)$ coincides with the Simpson compactification $\bS(X,2)$ because the structure sheaves of conics are stable sheaves.
In \S\ref{sec3.2}, we compare $\bH(X,d)$ and $\bS(X,d)$ for $d=3$.
By Theorem \ref{th1.2} (1), we have a morphism
\[ \bH(\PP^r,3)\lra \bS(\PP^r,3)\]
which is a smooth blow-up along the smooth locus $\Delta(\PP^r)$ of planar stable sheaves.
The inclusion $X\subset \PP^r$ induces an inclusion $\bS(X,3)\hookrightarrow \bS(\PP^r,3)$.
Similarly, the inclusion $X\subset \PP^r$ induces the inclusion map $\bH(X,3)\hookrightarrow \bH(\PP^r,3)$. Then by construction and direct calculation, the composition
$$\bH(X,3)\hookrightarrow \bH(\PP^r,3)\lra \bS(\PP^r,3)$$
factors through $\bS(X,3)$ so that we have a morphism
$\bH(X,3)\to \bS(X,3)$. Then we prove that the blow-up center $\Delta(\PP^r)$ intersects cleanly with $\bS(X,3)$ in $\bS(\PP^r,3)$ along the smooth locus $\Delta(X)$ of planar stable sheaves on $X$.
The meaning of \emph{the clean intersection} will be explained in the Definition-Proposition \ref{def-pro}.
Since $\bH(X,3)$ is the proper transform of $\bS(X,3)$ by definition, we conclude that the morphism $\bH(X,3)\to \bS(X,3)$ is the smooth blow-up along $\Delta(X)$.

The following diagram summarizes the comparison results for
a projective homogeneous variety $X\subset\PP^r$ and $d=3$:
\[
\xymatrix{ &&&\bM_3\ar[dl]_{\Gamma_2^3}\ar[dr]^{\Gamma_4^2}\\
&&\bM_2\ar[dl]_{\Gamma_1^2}&&\bM_4\ar[dr]^{\Gamma^3_5}\\
&\bM_1\ar[dl]_{\Gamma^1}&&&&\bM_5\ar[dr]^{\Gamma^1_6}&&\bH\ar[d]^{\Delta(X)}\\
\bM&&&&&&\bM_6\ar[r]^\cong & \bS. }\] All the arrows are blow-ups
and the blow-up centers are indicated above the arrows.

In \S\ref{sec4}, by using the blow-up formula of the cohomology groups (\cite{GriffithsHarris})
and the result of A. L\'opez-Mart\'in in \cite{Martin},
we calculate the Betti numbers of $\bH(X,d)$ and $\bS(X,d)$
when $d=2,3$ and $X=Gr(k,n)$ is any Grassmannian variety.

Quite recently, there has been strong interest in the Mori
theory of moduli spaces of curves. Since there are lots of
compactifications of the space of smooth curves, it is certainly a
good idea to give an order in the wild world of moduli spaces by Mori theory.
The most prominent result in this direction in recent years is the
following result of D. Chen.
\begin{theo} \cite{chen} When $X=\PP^3$ and $d=3$, $\bH(\PP^3,3)$ is a log flip of
 $\bM(\PP^3,3)$ with respect to $H+\alpha\Delta$, $-\frac{1}{5}< \alpha <0$ where $\Delta$ is the
boundary divisor and $H$ is the divisor of stable maps whose images intersect a fixed line in $\PP^3$.
\end{theo}
As shown in \cite{CK}, this flip is more precisely the composition
of three blow-ups and three blow-downs. Furthermore, we showed
that this result holds for any $\PP^r$ with $r\ge 3$ if we
replace $\bH$ by $\bS$. Note that when $X=\PP^3$, $\bH=\bS$.
We generalize this result to the case of arbitrary homogeneous projective varieties in this paper.

Another result in this line is due to D. Chen and I. Coskun as follows.
\begin{theo} \cite{chencoskun}
When $X=Gr(2,4)$ and $d=2$, $\bH$ is obtained from $\bM$ by a
blow-up followed by a blow-down.
\end{theo}
We will see below that this theorem is true for any projective
homogeneous variety $X$. See \cite{Kiem2} for more discussions on motivations.


\section{Preliminaries}
\subsection{Properties of a projective homogeneous variety}
In this subsection, we state all the properties of a projective homogeneous variety
which will be used to prove Theorems \ref{thm1}, \ref{thm2} and \ref{305}.
\begin{lemm}\label{assu5-1}
Let $X$ be a projective homogeneous variety with fixed embedding $i:X\hookrightarrow \PP^r$ and let $\cO_X(1)=i^*\cO_{\PP^r}(1)$. Then the following hold.
\begin{enumerate}
\item $H^1(\PP^1,f^*T_X)=0$ for any morphism $f:\PP^1\to X$.
\item $\mathrm{ev}:\cM_{0,1}(X,1)\to X$ is smooth where
$$\cM_{0,1}(X,1)=\{(f:\PP^1\to X, p\in
\PP^1)\,|\,\mathrm{deg} f^*\cO_X(1)=1 \}$$ is the moduli space of
$1$-pointed lines on $X$ and $\mathrm{ev}$ is the evaluation map at the marked point.
\item The moduli space $F_2(X)$ of planes in $X$ is smooth.
\item The defining ideal $I_X$ of $X$ in $\PP^r$ is generated by quadratic polynomials.
\end{enumerate}
\end{lemm}
\begin{proof}
Item (1) comes from the fact that the tangent bundle $T_X$ of $X$ is globally generated. Since the automorphism group
of $X$ acts transitively on itself, the generic smoothness of a morphism \cite[Corollary 10.7, III]{Hartshorne} implies item (2).
Items (3) and (4) are from \cite[Theorem 4.9]{Landsberg} and \cite{Ramanathan} respectively.
\end{proof}
\begin{defi}
A smooth projective variety $X$ is called \emph{convex} if item (1) above is satisfied.
\end{defi}
For most of our results, we will only need (1) and (2). However when we compare $\bH(X,3)$ and $\bS(X,3)$ in \S4.2, (3) and (4) will be useful.

\subsection{Deformations of morphisms and sheaves}
Let $Y$ be a projective curve with at worst nodal singularities and $X$ be a smooth projective variety. As we identify a map $f:Y\to X$ with its graph $G_f \subset Y\times X$ and thus a point in Hilbert scheme of $Y\times X$,
we have the following deformation theory of the morphism $f$.
\begin{prop}\label{rem1}\cite[Theorem 2.16]{Kollar}
The tangent space (resp. obstruction space) of $\Hom (Y,X)$ at a map $f:Y \lr X$ is
\[\Ext^{0}(f^*\Omega_X ,\cO_Y)\qquad (\mbox{resp. }\Ext^{1}(f^*\Omega_X ,\cO_Y)).\]
\end{prop}

If we allow $Y$ to vary, we have the following.
\begin{prop}\cite[Proposition 1.4, 1.5]{LiTian}\label{litian}
Let $f:C\to X$ be a point in the moduli space of stable maps $\cM_0(X,\beta)$ of genus $0$ to $X$ with homology class $\beta$. Then the tangent space (resp. the obstruction space) of $\cM_0(X,\beta)$ at $[f]$ is given by
\begin{equation}\label{2012}
\Ext^1_C([f^*\Omega_X\to\Omega_C],\cO_C)\qquad (\mbox{resp. }\Ext^2_C([f^*\Omega_X\to\Omega_C],\cO_C)),
\end{equation}
where $[f^*\Omega_X\to\Omega_C]$ is thought of as a complex of sheaves concentrated on the interval $[-1,0]$.
\end{prop}

When $X$ is a projective homogeneous variety, the obstruction space
$\Ext^2_C([f^*\Omega_X\to\Omega_C],\cO_C)$
is trivial because of item (1) in Lemma \ref{assu5-1} and the exact sequence
\[
H^1(f^*T_X)=\Ext^{1}(f^*\Omega_X ,\cO_C) \lr  \Ext^2 ([f^*\Omega_X \lr \Omega_C],\cO_C)\lr \Ext^2(\Omega_C,\cO_C)=0.
\]
Therefore, \'etale locally near a point $f$, $\cM_0(X,\beta)$ is isomorphic to a quotient
$$\Ext^1_C([f^*\Omega_X\to\Omega_C],\cO_C)/\mathrm{Aut}(f),$$
where $\mathrm{Aut}(f)$ is the automorphism group of the stable map $f$.

Deformation theory of stable sheaves is also well understood as follows.
\begin{prop}\cite[Corollary 4.5.2]{HL}
For a stable sheaf $E$ on a smooth projective variety $X$,
the tangent space (resp. the obstruction space) of
the Simpson moduli space $\cS imp^{P}(X)$ with fixed Hilbert polynomial $P$ is given by
\[
\Ext_X^1(E,E) \qquad (\mbox{resp. } \Ext_X^2(E,E)).
\]
\end{prop}

\subsection{Elementary modification of sheaves}
We recall the notion of destabilizing subsheaf
(resp. destabilizing quotient sheaf) of a pure sheaf (\cite[Chapter 2]{HL}).
For a fixed ample line bundle $\cO_X(1)$ on a smooth projective variety $X$,
let $$p (E):=\frac{\chi(E(m))}{r(E)}$$
be the reduced Hilbert polynomial of a pure sheaf $E$ on $X$
where $r(E)$ denotes the leading coefficient of the Hilbert polynomial $\chi(E(m))$ for $m>\!>0$.
Every pure sheaf has a unique filtration which is called the Harder-Narasimhan filtration.
\begin{defiprop}\cite[Theorem 1.3.4]{HL}
\begin{enumerate}
  \item For a pure sheaf $E$ on $X$, there exists a unique filtration of $E$
  \[ 0=E_0 \subset E_1 \subset E_2 \cdots \subset E_{k-1} \subset E_k =E,\]
  such that the reduced Hilbert polynomials decrease $p(E_i/E_{i-1})> p(E_{i+1}/E_i)$ and
  each quotient $E_{i+1}/E_i$ is semistable for any $i$.
  \item The first non-zero term $E_1$ (resp. the quotient $E/E_1$) is called as the destabilizing subsheaf (resp. the quotient sheaf) of $E$.
\end{enumerate}
\end{defiprop}
Note that the above theorem has a family version \cite[Theorem 2.3.2]{HL} and if there is a flat family of pure sheaves, a relative Harder-Narasimhan filtration exists.
\begin{exam}\label{25}
The destabilizing subsheaf (resp. the destabilizing quotient sheaf)
of the pure sheaf $\cO_L \oplus \cO_L(-1)$ on $\PP^2$ is $\cO_L$(resp. $\cO_L(-1)$) where
$L$ is a line in $\PP^2$.
\end{exam}
Now we introduce the notion of a modified sheaf which is originally introduced by Langton to prove the properness of
the moduli space of torsion free sheaves (\cite[Theorem 2.B.1]{HL}, \cite{langton}).
This is one of the main tools for constructing a morphism to the Simpson moduli space.
\begin{defi}
Let $\cE(X)$ be a flat family of sheaves on $X$ parameterized by a smooth variety $S$.
Let $Z$ be a smooth divisor of $S$ such that $\cE(X)|_Z$
has a flat family $\cA$ of destabilizing quotients.
Then
$$ elm_Z (\cE(X),\mathcal{A}):= \text{ker}\{\cE(X) \lr \cE(X)|_{X\times Z} \lr  \mathcal{A}\}$$
is called the elementary modification of sheaves $\cE(X)$ along $Z$.
\end{defi}
As we will see in Example \ref{exa1} below, the effect of elementary modification at the center $Z$ is the interchange of the sub and quotient sheaves.
\begin{exam}\label{exa1}
For a flat family of stable maps in $\PP^2$ of degree $2$
$$ \textbf{f}: \cC=\PP^1 \times \CC \lr \PP^2 \times \CC,\quad (s:t)\times (a) \mapsto (s^2:t^2:a st)\times (a),$$
let
$$\cE(\PP^2)=\textbf{f}_*\cO_{\cC}$$
be the direct image sheaf on $\PP^2 \times \CC$ which is flat over $\CC$
and let $Z=\{0\}$ be the origin of $\CC$. Then the central fiber $\cE(\PP^2)|_{\PP^2 \times \{0\}}$ fits into a short exact sequence
\[
\ses{\cO_L}{\cE(\PP^2)|_{\PP^2 \times \{0\}}}{\cO_L(-1)}
\]
where $L$ is the line $\{(z_0:z_1:0)\}$ in $\PP^2$.
Now let $A:= \cO_L(-1)$ which is the destabilizing quotient sheaf of $\cE(\PP^2)|_{\PP^2 \times \{0\}}$.
By direct calculation with local charts, it is straightforward that the central fiber of the modified sheaf is
\[elm_{\{0\}} (\cE(\PP^2),\cO_L(-1))|_{\PP^2 \times {\{0\}}}\cong\cO_{L^2},\]
where $L^2$ is the unique double line of $L$ in $\PP^2$ whose defining ideal is given by $<z_2^2>$.
Note that $\cO_{L^2}$ fits into the non-split short exact sequence
$$\ses{\cO_L(-1)}{\cO_{L^2}}{\cO_L}$$
and hence $\cO_{L^2}$ is stable.
\end{exam}


\section{Comparison result for $d=2$}\label{sec2}
In this section we relate the Kontsevich
compactification $\bM(X)=\bM(X,2)$ with the Simpson
compactification $\bS(X)=\bS(X,2)\cong \bH(X,2)$ in terms of explicit blow-ups.
Our goal is to generalize Theorem 4.1 and Proposition 4.3 in \cite{Kiem} to projective homogeneous varieties.
Throughout this section, we only use the property (1) of Lemma \ref{assu5-1}.
In \S\ref{sec2.1}, we blow up $\bM(X)$ and apply elementary modification of sheaves to construct a family of stable sheaves on $X$ which gives rise to a morphism to $\bS(X)$.
In \S\ref{sec2.2}, we show that the morphism to $\bS(X)$ is in fact a blow-up. 


\subsection{Blow-ups}\label{sec2.1}
To avoid singularities, we express $\bM(X)$ as an $SL(2)$-quotient of a smooth variety $P_1(X)$
and construct a family of stable sheaves on $X$ parameterized by a blow-up $P_2(X)$ of $P_1(X)$
via elementary modification. In this way we obtain an invariant morphism $P_2(X)\to \bS(X)$
which induces a birational morphism $P_2(X)/SL(2)\to \bS(X)$. By \cite{Kirwan}, $P_2(X)/SL(2) \to P_1(X)/SL(2)=\bM(X)$ is a blow-up.

Let
$$P_0 :=\PP(\Sym^2\CC^2 \otimes \CC^{r+1})$$
be the projective space where $SL(2)$ acts on $\Sym^2\CC^2$ in the standard fashion and trivially on $\CC^{r+1}$.
An element of $P_0$ can be thought of as an $(r+1)$-tuple of quadratic polynomials in two variables up to constant multiple.
Let $P_0^{ss}$ denote the semistable part of $P_0$ and let $\Sigma^k\subset P_0^{ss}$
be the locus of tuples with $k$ common zeros so that we have a disjoint union
$$P_0^{ss} =\Sigma^0 \cup \Sigma^1 \cup \Sigma^2.$$
Let $P_1$ be the blow-up of $P_0^{ss}$ along $\Sigma^2$
and let $\rho:P_1 \lr P_0^{ss}$ be the blow-up morphism
with the exceptional divisor $E$.
Then the set $P_1^{ss}$ of the semistable points coincides
with the set $P_1^s$ of the stable points because the
strictly semistable points disappear
with respect to the linearization on the line bundle
$\rho^* \cO(1) \otimes \cO(-\epsilon E),0<\epsilon <\!< 1$ (\cite[\S 6]{Kirwan}).
By modifying some tautological family of rational maps over $P_1^s$ which is given by the evaluation morphism,
the third named author constructed a family of stable maps to $\PP^1\times\PP^r$
over $ P_1 (\PP^r):=P_1^s $ (\cite[Proof of Theorem 4.1]{Kiem})
\begin{equation}\label{eq-1-1}
 \xymatrix{\widetilde{\cC} \ar[d]\ar[r]& \PP^1\times\PP^r \\
 P_1(\PP^r).& \\}
\end{equation}
By composing \eqref{eq-1-1} with the projection $\PP^1\times \PP^r\to \PP^r$ and stabilizing $\widetilde{\cC}$, we obtain a family of stable maps
\begin{equation}\label{eq-1}
 \xymatrix{{\cC} \ar[d]\ar[r]& \PP^r \\
 P_1(\PP^r)& \\}
\end{equation}
and hence an $SL(2)$-invariant morphism
\begin{equation}\label{eq-2}
P_1 (\PP^r) \lr \bM(\PP^r).
\end{equation}
Finally he showed that $P_1 (\PP^r)/SL(2)\cong  \bM(\PP^r)$
by using Zariski's main theorem (\cite[\S9, III]{Mumford}).
By Luna's slice theorem (\cite[Appendix 1.D]{Mumford2}),  $P_1(\PP^r)$ is a principal bundle over $\bM(\PP^r)$ in the \'etale sense.
Furthermore, by \eqref{eq-1-1}, we have an injective morphism
\begin{equation}\label{1103231} P_1(\PP^r)\lra \cM_0(\PP^1\times\PP^r,(1,2))\end{equation}
to the moduli space of stable maps to $\PP^1\times \PP^r$ of genus $0$ and bidegree $(1,2)$.
By the construction of \eqref{eq-1-1} in \cite{Kiem}, the morphism in \eqref{1103231} factors through the open subvariety of the moduli space $\cM_0(\PP^1\times\PP^r,(1,2))$ consisting of stable maps whose automorphism groups are trivial.
Since $\PP^1\times\PP^r$ is convex, this open subvariety is smooth by \cite[Theorem 2]{FP}.
As \eqref{1103231} is an injective morphism of smooth varieties which is an isomorphism on the open locus of nonsingular conics $\rho^{-1}(P_0^s)$, we find that \eqref{1103231} is an open immersion by \cite[II.4. Theorem 2]{Shafarevich}.

For a projective homogeneous variety $X\subset\PP^r$, we consider the fiber products
\begin{equation}\label{eq-3}
P_1(X):=\bM(X)\times_{\bM(\PP^r)}P_1 (\PP^r),
\end{equation}
$$
\cC_X :=P_1(X) \times_{P_1 (\PP^r)}\cC,
$$
from \eqref{eq-2}, \eqref{eq-1} and the inclusion $\bM(X)\hookrightarrow \bM(\PP^r)$. By definition, it is obvious that we have a Cartesian square of open immersions
$$\xymatrix{ P_1(X)\ar@{^(->}[r]\ar@{^(->}[d] & \cM_0(\PP^1\times X,(1,2))\ar@{^(->}[d]\\
P_1(\PP^r) \ar@{^(->}[r] &\cM_0(\PP^1\times\PP^r,(1,2)).}$$
The image of $P_1(X)$ in $\cM_0(\PP^1\times X,(1,2))$ is contained in the open locus of stable maps
with no non-trivial automorphisms.
Since this locus is smooth by the convexity of $\PP^1\times X$ (\cite[Theorem 2]{FP}), $P_1(X)$ is also a \emph{smooth} quasi-projective variety.



On the other hand, since $P_1(\PP^r)$ is a principal bundle over $\bM(\PP^r)$, $P_1(X)$ is $SL(2)$-invariant and $P_1(X)/SL(2)\cong \bM(X)$.
Moreover, there exists an induced family of stable maps to $\PP^r$ over $P_1(X)$
all of which factor through $X$ so that we get a diagram
$$
 \xymatrix{\cC_X \ar[d]_{\pi}\ar[r]^{ev}& X \\
 P_1 (X).&\\
}
$$

To define a rational map from $P_1(X)$ to $\bS(X)$, we consider the morphism
$$
(ev,\pi):\cC_X \lr X \times P_1 (X).
$$
The direct image sheaf $\cE_0(X) := (ev,\pi)_*\cO_{\cC_X}$ is
a family of coherent sheaves on $X$, flat over $P_1(X)$
because the Hilbert polynomial is constantly $2m+1$ and $P_1(X)$ is a reduced scheme
(\cite[Theorem 9.9, III]{Hartshorne}).
By Lemma \ref{dirstab} below,
$\cE_0(X)|_{X\times \{z\}}$ is a stable sheaf for each closed point $z \in P_1(X)$
which gives rise to a nonsingular conic.
Hence there exists a rational map
\begin{equation}\label{eq-4}
\phi :P_1(X) \dashrightarrow \bS(X)
\end{equation}
by the universal property of $\bS(X)$. By definition, $\phi$ is $SL(2)$-invariant and
thus we have an induced birational map
\begin{equation}\label{eq-5}
\overline{\phi} :\bM(X) \dashrightarrow \bS(X).
\end{equation}

Next we find the undefined locus of $\phi$
and then blow up $P_1(X)$ along the locus.
\begin{lemm}\label{dirstab} \cite[Proposition 3.18]{chencoskun}
For $r\ge 3$, let
$f: C \lra X\subset \PP^r$ be a stable map of genus $0$ and degree $d\le 3$.
Then the direct image sheaf $f_* \cO_C $ is stable
if $f$ is not a multiple cover
(i.e. no component of the image $f(C)$ is multiply covered by $f$).
\end{lemm}
For $d=2$, if $f:C\to X$ is a multiple cover,
then the image $f(C)$ has to be a line $L$ in $X$ and $f_*\cO_C\cong \cO_L\oplus \cO_L(-1)$ which is unstable.
Therefore the undefined locus of
the birational map $\phi$ in \eqref{eq-4}
is exactly the locus $\Theta^1 (X)$ of stable maps whose image is a line in $X$.
When $X=\PP^r$,
let us use the natural inclusion
\[
\bM(\PP\cU) \hookrightarrow \bM(\PP^r)
\]
where $\cU$ is the universal rank $2$ bundle over $Gr(2,r+1)$
and $\bM(\PP\cU)$ denotes the relative moduli space of stable maps of degree $2$ to the fibers of $\PP\cU\to Gr(2,r+1)$.
Let
$$\Theta^1 (\PP^r):= \bM(\PP\cU)\times_{\bM(\PP^r)} P_1(\PP^r).$$
If we fix a line $L$ in $\PP^r$ or an inclusion $\CC^2\hookrightarrow\CC^{r+1}$, we have an inclusion
\[ \PP(\Sym^2(\CC^2)\otimes \CC^2)\hookrightarrow
\PP(\Sym^2(\CC^2)\otimes \CC^{r+1})\]
and thus $P_1(\PP^1)\hookrightarrow P_1(\PP^r)$.
This means that $\Theta^1 (\PP^r)$ is a $P_1(\PP^1)$-bundle over $Gr(2,r+1)$.
For a general homogeneous variety $X \subset \PP^r$,
let $F_1(X)$ be the variety of lines in $X$ which is smooth by item (1) of Lemma \ref{assu5-1}.
Let
\[
\Theta^1 (X):= F_1(X) \times_{Gr(2,r+1)} \Theta^1 (\PP^r)
\]
be the fiber product where $F_1(X)\hookrightarrow F_1(\PP^r)=Gr(2,r+1)$ is the natural inclusion.
Then $\Theta^1(X)$ is also a $P_1(\PP^1)$-bundle over $F_1(X)$.
In particular, $\Theta^1(X)$ is a smooth subvariety of $P_1(X)$.
Let $$\Gamma^1(X):=\Theta^1 (X)/SL(2),$$ which is a $\PP^2$-bundle over $F_1(X)$
because $P_1(\PP^1)/SL(2)=\bM(\PP^1)\cong \PP^2$ parameterizes choices of two branch points.
By Remark \ref{rem1.1}, $\Gamma^1(X)$ is in fact a disjoint union of irreducible components $\Gamma^1(X,\beta)\subset \bM(X,\beta)$
where $\Gamma^1(X,\beta)$ is a $\bM(\PP^1)$-bundle over the moduli space $F(X,\frac{\beta}{2})$
of lines in $X$ such that $i_*\frac{\beta}{2}=[\PP^1]\in H_2(\PP^r,\ZZ)$.

For $[f]\in \Theta^1 (X)$ representing a stable map $f:C \lr L \subset X$,
\begin{equation}\label{eq-6}
f_* \cO_C \cong \cO_L \oplus \cO_L(-1)
\end{equation}
as an $\cO_X$-module
where $L$ is a line in $X$ (cf. \cite[Lemma 4.5]{CK}).
To extend the birational map $\overline{\phi}$ in \eqref{eq-5}
we apply a blow-up and an elementary modification of sheaves.
Let
$$
q:P_2(X)\lr P_1(X)
$$
be the blow-up of $P_1(X)$ along $\Theta^1 (X)$.
Let $\Theta_1^1 (X)$ be the exceptional divisor of $\Theta^1 (X)$,
let $\Gamma_1^1(X):=\Theta_1^1 (X)/SL(2)$, and $$\bM_1(X):=P_2(X)/SL(2).$$
The destabilizing quotients $f_*\cO_C\to \cO_L(-1)$ of \eqref{eq-6} form a flat family $\mathcal{A}_1$ over $\Theta^1_1(X)$ by the relative Harder-Narasimhan filtration (\cite[Chapter 2]{HL}).
Let
$$\cE_1(X):= elm_{\Theta_1^1 (X)} ((1_X \times q )^* \cE_0 (X),\mathcal{A}_1)$$
be the elementary modification of the pull-back of $\cE_0(X)$ with respect to $\cA_1$.
For each $z \in \Theta_1^1 (X)$, $\mathcal{A}_1|_{X\times \{z\}} = \cO_L(-1)$
if $q(z)$ represents a stable map $f:C \lr L\subset X$.
\begin{prop}\label{pro1}
$\cE_1(X)$ is stable for every point in $\Theta_1^1(X)$. Hence
there exists a birational morphism
$$p:\bM_1(X) \lr \bS(X),$$
which extends the rational map
$\overline{\phi}:\bM(X)\dashrightarrow \bS(X)$ in \eqref{eq-5}.
\end{prop}
\begin{proof}
We must show that $\cE_1(X)|_{X\times \{z\}}$ is stable
when $q(z)$ represents  stable map $f:C \lr L \subset X$ where $L$ is a line.
It is well known that the effect of elementary modification is
the interchange of the sub and quotient sheaves (cf. Example \ref{exa1}).
In our case, we claim that $\cE_1(X)|_{X\times \{z\}}$ fits into a non-split short exact sequence
$$ \ses{\cO_L(-1)}{\cE_1(X)|_{X\times\{z\}}}{\cO_L}$$
and therefore it is stable. We will prove this claim by studying the Kodaira-Spencer map of sheaves as follows (cf. \cite[Chapter 10.1]{HL}).

Choosing a vector $v$ in
$$T_{q(z)}P_1(X)=Hom_{\CC}(\mathrm{Spec}\CC[\epsilon]/(\epsilon^2), P_1(X))$$
is equivalent to having a flat family of stable maps over $\mathrm{Spec}\CC[\epsilon]/(\epsilon^2)$
$$ \tilde f: \tilde C:=C\times\mathrm{Spec}\CC[\epsilon]/(\epsilon^2)\lr X $$
whose central fiber is the stable map $f:C \lr L \subset X$.
Then
\begin{equation}\label{eq-100}
\cE_0(X)|_{X\times \mathrm{Spec}\CC[\epsilon]/(\epsilon^2)}= \tilde f_* \cO_{\tilde C}
\end{equation}
on $\tilde X := X\times \mathrm{Spec}\CC[\epsilon]/(\epsilon^2)$.
The elementary modification $\cE_1(X)|_{\tilde X}$
fits into the following diagram of $\cO_{\tilde X}$-modules
\[
\xymatrix{
0 \ar[r] & {\epsilon\cdot f_*\cO_C} \ar[r]\ar@{=}[d] & \cE_1(X)|_{\tilde X}
\ar@{-->}[d]\ar[r]
&\cO_L \ar[r] \ar[d] &0\\
0\ar[r] &\epsilon\cdot f_*\cO_C\ar[r] & \cE_0(X)|_{\tilde X}\ar[r]
&f_*\cO_C \ar[r]&0,
}
\]
where the right vertical map comes from \eqref{eq-6}
and the last term in the second row is $0$ because $R^1f_*\cO_C =0$.
Computing the central fiber
$$\cE_1(X)|_{\tilde X } /\epsilon \cdot\cE_1(X)|_{\tilde X}$$
amounts to calculating the push-out diagram
\[
\xymatrix{
0 \ar[r] & \epsilon\cdot \cO_L(-1) \ar[r] & \cE_1(X)|_{\tilde X}/\epsilon \cdot\cE_1(X)|_{\tilde X}\ar[r]\ar@{<--}[d]
&\cO_L \ar[r] &0\\
0\ar[r] &\epsilon\cdot f_*\cO_C\ar[r] \ar[u] & \cE_1(X)|_{\tilde X}\ar[r]
&\cO_L \ar[r] \ar@{=}[u]&0,
}
\]
where the first vertical map comes from \eqref{eq-6} again.
These operations are represented by $\CC$-linear maps
\begin{equation}\label{1103234}
KS: T_{q(z)}P_1(X) \lr  \Ext^1_{X}(f_*\cO_C,f_*\cO_C) \lr \Ext^1_{X}(\cO_L, f_*\cO_C) \lr \Ext^1_{X}(\cO_L,\cO_L(-1))
\end{equation}
whose composition sends $ v\mapsto \cE_1(X)|_{\tilde X } /\epsilon \cdot\cE_1(X)|_{\tilde X}$.
The first map of $KS$ is exactly the Kodaira-Spencer map of sheaves \cite[Chapter 10.1]{HL}
and the others come from \eqref{eq-6}.

On the other hand, since $P_1(X)$ is open in $\cM_0(\PP^1\times X,(1,2))$,  its tangent space $T_{P_1(X),q(z)}$
at $q(z)=(f:C \lr L \subset X)$ is isomorphic to
\[ \Ext^1([f^*\Omega_{\PP^1\times X}\to \Omega_C], \cO_C) \]
which fits into the exact sequence
\begin{equation}\label{1103232} 0\lra \Ext^0(\Omega_C,\cO_C) \lra H^0(f^*T_{\PP^1\times X})\lra T_{P_1(X),q(z)}\lra \Ext^1(\Omega_C,\cO_C)\lra 0.\end{equation}
Likewise, the tangent space to the fiber $P_1(\PP^1)=P_1(L)$ of $\Theta^1(X)\to F_1(X)$ over $L$ is isomorphic to
\[ \Ext^1([f^*\Omega_{\PP^1\times L}\to \Omega_C], \cO_C) \]
which fits into the exact sequence
\begin{equation}\label{1103233} 0\lra \Ext^0(\Omega_C,\cO_C) \lra H^0(f^*T_{\PP^1\times L})\lra T_{P_1(L),q(z)}\lra \Ext^1(\Omega_C,\cO_C)\lra 0.\end{equation}
From \eqref{1103232} and \eqref{1103233}, we find that
\[ T_{P_1(X),q(z)}/T_{P_1(L),q(z)}\cong H^0(f^*N_{L/X})\cong H^0(N_{L/X})\oplus H^0(N_{L/X}(-1))\]
by the projection formula, where $N_{L/X}$ denotes the normal bundle of $L$ in $X$.
Since the tangent space to $F_1(X)$ is $H^0(N_{L/X})$, by taking further quotient by $H^0(N_{L/X})$, we obtain
\begin{equation}\label{1103236} N_{\Theta^1(X)/P_1(X),q(z)}=H^0(N_{L/X}(-1))\cong \Hom_X(I_{L/X},\cO_L(-1))\end{equation}
by $N_{L/X}^\vee=I_{L/X}/I_{L/X}^2$ where $I_{L/X}$ is the ideal sheaf of $L$ in $X$.

Obviously moving in $P_1(L)$ doesn't change the sheaf $f_*\cO_C=\cO_L\oplus \cO_L(-1)$ and the deformation space $H^0(N_{L/X})\cong \Ext^1_X(\cO_L,\cO_L)$ of $L$ is mapped to zero by the last arrow of \eqref{1103234}. Therefore the map $KS$ descends to an isomorphism
\begin{equation}\label{eq-7}
\overline{KS}:N_{\Theta^1(X)/P_1(X),q(z)}=\Hom_{X}(I_{L/X},\cO_L(-1)) \mapright{\delta} \Ext^1_{X}(\cO_L,\cO_L(-1))
\end{equation}
which is exactly the coboundary map $\delta$ of the short exact sequence $$\ses{I_{L/X}}{\cO_X}{\cO_L}$$ by direct inspection (cf. \cite[Lemma 4.6]{CK}).
Of course, $\delta$ is an isomorphism because $H^i(\cO_L(-1))=0$ for $i=0,1$.

In summary, the image $\overline{KS}(v)=\cE_1(X)|_{X\times \{z\}}$ of $v\neq 0$ for $z\in \Theta_1^1(X)$ is exactly a non-split extension class in $\Ext^1_{X}(\cO_L,\cO_L(-1))$ and thus it is stable.
Hence there exists a morphism $P_2(X)\lr \bS(X)$ by the universal
property of $\bS(X)$, which is $SL(2)$-invariant by construction.
Therefore $P_2(X)\lr \bS(X)$ descends to a birational morphism $\bM_1(X) \lr \bS(X)$.
\end{proof}

\begin{rema}\label{rem3.5} Since, for any double covering map $f:\PP^1 \lr L\subset X$ of a line $L$,
$$H^1(f^*T_X)=H^1(T_X \otimes f_*\cO_{\PP^1})=H^1(T_X|_L \oplus T_X|_L(-1))=0 $$
by the projection formula and item (1) in Lemma \ref{assu5-1}, we have $H^1(T_X|_L(-1))=0$ and hence $H^1(N_{L/X}(-1))=0$. By Riemann-Roch,
  \[\dim \Hom_{X}(I_{L/X},\cO_L(-1))=\dim H^0(N_{L/X}(-1)) = \int_{\frac{\beta}{2}} c_1 (T_X)-2\]
when $i_*\frac{\beta}{2}=[\PP^1]\in H_2 (\PP^r,\ZZ)$.
Hence the linear image locus $\Gamma^1(X,\beta)$ in the irreducible component $\bM(X, \beta)$
is a pure dimensional subvariety.
\end{rema}


Next we claim that the blow-up morphism $$q:P_2(X)\lr P_1(X)$$
is just the proper transform of $P_1(X)$ by the blow-up morphism $P_2(\PP^r)\lr P_1(\PP^r)$.
\begin{defiprop}\cite[Lemma 5.1]{LLi}\label{def-pro}
Let $A$ and $B$ be smooth closed subvarieties of a nonsingular variety $P$ and let
$U$ be the set-theoretic intersection of $A$ and $B$ (i.e. $U$ is the reduced scheme of the fiber product $A\times _P B$). Suppose $U$ is also smooth. If
\[T_u U= T_u A \cap T_u B \]
for all $u\in U$, then we call $A$ and $B$ \emph{intersect cleanly} along $U$ in $P$. Then the following hold.
\begin{enumerate}
  \item $U$ is the scheme theoretic intersection in the sense that
  $I_A + I_B =I_U.$
  \item  The smooth blow-up of $A$ along $U$ is just the proper transformation of $A$ along the smooth blow-up morphism $bl_B P \lr P$.
\end{enumerate}
\end{defiprop}
\begin{lemm}\label{trans1}
$P_1(X)$ intersects with $\Theta^1(\PP^r)$ cleanly
along $\Theta^1(X)$ in $P_1(\PP^r)$. Hence $P_2(X)$ is the proper transform of $P_1(X)$ via the blow-up $P_2(\PP^r)\to P_1(\PP^r)$.
\end{lemm}
\begin{proof}
Clearly, set theoretic intersection $P_1(X)$ and $\Theta^1(\PP^r)$ in $P_1(X)$ is $\Theta^1(X)$
because of the universal property of the fiber product.
Moreover recall that $\Theta^1(X)$ is a $P_1(\PP^1)$-bundle over $F_1(X)$ and thus $\Theta^1(X)$ is smooth.
On the other hand, by \eqref{1103236}, the inclusion $N_{L/X}\subset N_{L/\PP^r}$ induces an inclusion
\begin{equation}\label{2022}
N_{\Theta^1(X)/P_1(X),q(z)}\cong H^0(N_{L/X}(-1))\subset H^0(N_{L/\PP^r}(-1))\cong N_{\Theta^1(\PP^r)/P_1(\PP^r),q(z)},
\end{equation}
where $ z\in \Theta_1^1(X)$ and $q: P_2(X) \lr P_1(X) $ is the blow-up morphism.
From the commutative diagram of exact sequences
\[
\xymatrix{
0\ar[r]&T_{q(z)}\Theta^1(X)\ar[r]\ar@{^(->}[d]& T_{q(z)} P_1(X)\ar[r]\ar@{^(->}[d]& N_{\Theta^1(X)/P_1(X),q(z)}\ar[r]\ar@{^(->}[d]&0\\
0\ar[r]&T_{q(z)}\Theta^1(\PP^r)\ar[r]& T_{q(z)} P_1(\PP^r)\ar[r]& N_{\Theta^1(\PP^r)/P_1(\PP^r),q(z)}\ar[r]&0,\\
}
\]
and the injectivity of the last vertical arrow \eqref{2022}, we find immediately that
$$T_{q(z)}\Theta^1(X)= T_{q(z)} P_1(X)\cap T_{q(z)}\Theta^1(\PP^r).$$
The lemma now follows from Definition-Proposition \ref{def-pro}.
\end{proof}


\subsection{Blow-down}\label{sec2.2}
We show that the birational morphism $ p: \bM_1(X) \lr \bS(X)$
in Proposition \ref{pro1} is a smooth blow-up morphism
by analyzing a neighborhood of the exceptional divisor
$\Gamma_1^1(X)=\Theta_1^1 (X)/SL(2)$.
Let $\Gamma^1_2(X)=p(\Gamma^1_1(X))$.
\begin{prop}
 $p:\bM_1(X) \lr \bS(X)$ is the smooth blow-up morphism along $\Gamma^1_2(X)$.
\end{prop}
\begin{proof}
If $f:C\to L\subset X$ is represented by an element in $\Theta^1(X)$, the automorphism group is $\ZZ_2$ and thus
$\bM(X)$ has $\ZZ_2$-quotient singularities along the blow-up center
$\Gamma^1 (X)$ by Proposition \ref{litian}. Therefore if we blow up $\bM(X)$ along
$\Gamma^1 (X)$, then the singularity is resolved (\cite[\S3]{Kiem})
and hence $\bM_1 (X)$ is smooth. We have seen that $\Gamma^1(X)$ is a $\PP^2$-bundle over $F_1(X)$ and the normal bundle to $\Theta^1(X)$ is
independent of the $\PP^2$-directions by \eqref{eq-7}. Therefore the exceptional divisor $\Gamma^1_1(X)$ in $\bM_1(X)$ is a $\PP^2$-bundle
over a $\PP(\Ext_X^1(\cO_L,\cO_L(-1)))$-bundle over $F_1(X)$. By the Fujiki-Nakano criterion \cite{Fujiki}, it suffices to show that \begin{enumerate}
\item $p:\Gamma_1^1(X)\to \Gamma_2^1(X)$ is a projective bundle with fiber $\PP^2$;
\item the restriction of the normal bundle of $\Gamma_1^1(X)$ to each fiber $\PP^2$ is $\cO_{\PP^2}(-1)$.\end{enumerate}

Now item (1) is a direct consequence of our proof of Proposition \ref{pro1}. Note that the $\PP^2$ direction in $\Gamma^1_1(X)$ tells us only about the double cover of the image line $L$ while the $\PP N_{\Theta^1(X)/P_1(X),q(z)}$ direction gives all distinct extension sheaves of $\cO_L$ by $\cO_L(-1)$.

When $X=\PP^r$, this proposition was proved in \cite{Kiem}.
For $X\subset \PP^r$, by Lemma \ref{trans1}, the normal bundle of $\Gamma^1_1(X)$ in $\bM_1(X)$ is the restriction of that of $\Gamma^1_1(\PP^r)$ in $\bM_1(\PP^r)$. Therefore, we see that (2) holds for $X$ as well.
\end{proof}

In summary, we have a blow-up/down diagram
which generalizes Theorem 4.1 and Proposition 4.3 in \cite{Kiem}.
\begin{theo}\label{thm1} For a projective homogeneous variety $X$ in $\PP^r$,
$\bM(X)=\bM(X,2)$ and $\bS(X)=\bS(X,2)$ are related by blow-ups as follows:
$$
\xymatrix{ & \bM_1(X)\ar[ld]_{\Gamma^1(X)} \ar[rd]^{\Gamma^1_2(X)}& \\
\bM(X)&&\bS(X).}
$$
Here $\Gamma^1(X)$ and $\Gamma^1_2(X)$ indicate the blow-up centers.
\end{theo}
\begin{rema}
Since we used only item (1) in Lemma \ref{assu5-1} to prove Theorem \ref{thm1},
Theorem \ref{thm1} holds for any convex variety $X$.
\end{rema}

\section{Comparison results for $d=3$} \label{sec3}
Let $X$ be a projective homogeneous variety over $\CC$ with fixed embedding $i:X\hookrightarrow \PP^r$.
In \S\ref{sec3.1}, we will use properties (1) and (2) of Lemma \ref{assu5-1} only.
But, in \S\ref{sec3.2}, we will use all items of Lemma \ref{assu5-1}.
In this section we fix $d=3$
and compare the compactifications $\bM(X),\bS(X)$
and $\bH(X)$ by sequences of blow-ups. Let $L$ be a line in $X$
and let $$t:=\dim \Ext_X^1(\cO_L, \cO_L(-1))$$
be the dimension of the moduli space $F_1(X,x)$ of lines
which pass through a given point $x$ in $X$ (cf. \cite[Theorem 1.7, II]{Kollar}).
Note that $t$ depends only on $\beta$ such that $i_*\frac{\beta}{3}=[\PP^1] \in H_2 (\PP^r,\ZZ)$ (cf. Remark \ref{rem3.5}).
\subsection{Comparison of $\bM(X)$ and $\bS(X)$}\label{sec3.1}
In this subsection we will generalize the comparison result \cite[Theorem 1.4]{CK} to arbitrary homogeneous projective varieties. The strategy is the same as in the degree 2 case above:
\begin{enumerate}\item Blow up components of the locus of unstable sheaves.
\item Apply elementary modification to make sheaves stable.
\item Analyze neighborhoods of the exceptional divisors to factorize the morphism to $\bS(X)$.
\end{enumerate}
We will use only (1) and (2) of Lemma \ref{assu5-1} in this subsection.

As in \S\ref{sec2}, we begin with a description of $\bM(X)$ as the GIT quotient of a smooth quasi-projective variety.
\begin{theo}\cite{Parker} $\bM(\PP^r)$ is the GIT quotient of the moduli space $$\cM_0(\PP^1\times\PP^r,(1,3))$$ of stable maps to $\PP^1\times\PP^r$ of genus $0$ and bidegree $(1,3)$ by $SL(2)$ with respect to a suitable linearization. Here the action of $SL(2)$ on $\cM_0(\PP^1\times\PP^r,(1,3))$ is induced from the standard action on $\PP^1$ and trivial action on $\PP^r$.
\end{theo}

By \cite{Parker}, there are no strictly semistable points on $\cM_0(\PP^1\times\PP^r,(1,3))$. Let $Q_0(\PP^r)$ be the stable part of $\cM_0(\PP^1\times\PP^r,(1,3))$ so that
\[ Q_0(\PP^r)/SL(2) \cong \bM(\PP^r).\]
Moreover, by \cite[Lemma 5.2]{KM}, the stable part $Q_0(\PP^r)$ is contained in the open subvariety of $\cM_0(\PP^1\times\PP^r,(1,3))$ of stable maps whose automorphism groups are trivial. Hence $Q_0(\PP^r)$ is smooth by the convexity of $\PP^1\times \PP^r$ (\cite[Theorem 2]{FP}).
In fact, $Q_0(\PP^r)$ is isomorphic to
the smooth quasi-projective variety $\mathbf{P}_5$ in Proposition 5.6 of \cite{KM}
by its construction.
By composing the universal family
\[\xymatrix{
\widetilde{\cC}\ar[r]\ar[d] & \PP^1\times\PP^r\\
Q_0(\PP^r)
}\]
with the projection $\PP^1\times\PP^r\to \PP^r$ and stabilizing the domain curves, we
obtain a family of stable maps to $\PP^r$
\[\xymatrix{
\cC\ar[r]\ar[d] & \PP^r\\
Q_0(\PP^r)
}\]
which induces the quotient morphism $Q_0(\PP^r)\to \bM(\PP^r)$.

Let $Q_0(X)$ be the fiber product
$$
Q_0(X):=Q_0(\PP^r)\times_{\bM(\PP^r)}\bM(X)\subset Q_0(\PP^r)
$$
so that $Q_0(X)/SL(2)\cong \bM(X)$. Then via the inclusion
\[ \cM_0(\PP^1\times X,(1,3))\hookrightarrow \cM_0(\PP^1\times\PP^r,(1,3)),\]
we find that $Q_0(X)$ is the stable part of the moduli space $\cM_0(\PP^1\times X,(1,3))$ of stable maps to $\PP^1\times X$ which is \emph{smooth}
by the convexity of $\PP^1\times X$ as before.



Let $\cC_X=\cC\times_{Q_0(\PP^r)}Q_0(X)$ so that we have an induced family of stable maps
\[
 \xymatrix{\cC_X \ar[d]_{\pi}\ar[r]^{\mathrm{ev}}& X\\
 Q_0(X)& \\}
\]
which gives us a rational map
\[
\phi: Q_0(X) \dashrightarrow \bS(X)
\]
defined by the family of coherent sheaves
\[
\cE_0(X):= (\mathrm{ev},\pi)_* \cO_{\cC_X}
\]
on $X$ parameterized by $Q_0(X)$.

By Lemma \ref{dirstab},
the locus of unstable sheaves in the family $\cE_0(X)$ consists of two
subvarieties of $\bM(X)$;
\begin{enumerate}
\item the locus $\Gamma_0^1(X)$ of stable maps
whose images are lines,
\item the locus $\Gamma_0^2(X)$ of stable maps
whose images are unions of two lines.\end{enumerate}
These loci can be also described as GIT quotients
by using the descriptions for $\PP^r$
in \cite[\S4.2]{CK}. It was proved that $\Gamma_0^1(\PP^r)$
is isomorphic to $\Theta_0^1(\PP^r)/SL(2)$ where $\Theta_0^1(\PP^r)$ is a
$\PP(\Sym^3\CC^2\otimes\CC^2)^s$-bundle over $Gr(2,r+1)$
where $\PP(\Sym^3\CC^2\otimes\CC^2)^s$ denotes the stable part of $\PP(\Sym^3\CC^2\otimes\CC^2)$
with respect to the $SL(2)$ action which is standard on $\Sym^3\CC^2$ and trivial on $\CC^2$.
For general $X\subset\PP^r$, using the natural injection
$F_1(X) \hookrightarrow F_1(\PP^r)= Gr(2,r+1)$ of the varieties of lines,
we let $\Theta_0^1(X)$ be the fiber product
$$
\Theta_0^1(X):=F_1(X) \times_{Gr(2,r+1)}\Theta_0^1(\PP^r).
$$
Then we obviously have
$$\Gamma_0^1(X)=\Theta_0^1(X)/SL(2).$$
Note that $\Gamma^1_0(X)$ is an $\bM(\PP^1)$-bundle over $F_1(X)$
where $$\bM(\PP^1)=\cM_0(\PP^1,3)=
\PP(\Sym^3\CC^2\otimes\CC^2)^s/SL(2)=
\PP(\Sym^3\CC^2\otimes\CC^2)\git SL(2)$$ is the moduli space of stable maps to $\PP^1$ of genus $0$ and degree $3$.
Also, $\Gamma_0^1(X)$ is a disjoint union of $\Gamma_0^1(X, \beta)$ by Remark \ref{rem1.1} where
$\Gamma_0^1(X, \beta)$ is a $\bM(\PP^1)$-bundle over the moduli space $F_1(X, \frac{\beta}{3})$ of lines in $X$
with $i_*\frac{\beta}{3}=[\PP^1] \in H_2(\PP^r,\ZZ)$. We let $F_1(X,\frac{\beta}{3})=\emptyset$ if $\frac{\beta}{3}\notin H_2(X,\ZZ)$.

For $X=\PP^r$, it was proved in \cite{CK} that
\[\Gamma_0^2(\PP^r)=\Theta_0^2(\PP^r)/SL(2),\]
where $\Theta_0^2(\PP^r)$ is a $\PP^{r-1}$-bundle over a smooth variety $\bB(\PP^r)$.
Here $\bB(\PP^r)$ is a $[ \PP^1 \times \PP(\Sym^2 \CC^2 \otimes \CC^2)]^s$-bundle over $Gr(2,r+1)$
where $[ \PP^1 \times \PP(\Sym^2 \CC^2 \otimes \CC^2)]^s$ is the stable part of $\PP^1 \times \PP(\Sym^2 \CC^2 \otimes \CC^2)$
with respect to an $SL(2)$ action.
Moreover $\bB(\PP^r)/SL(2)$  is the moduli space $\cM_{0,1}(\PP \cU, 2 )$ of relative stable maps of degree $2$ with one marked point
where $\cU$ is the universal rank $2$ bundle over $Gr(2,r+1)$. See \cite[\S4.2]{CK} for more details.

For the projective homogeneous variety $X \subset \PP^r$, let
\[
\bB(X):= F_1(X) \times_{Gr(2,r+1)} \bB(\PP^r)
\]
be the fiber product which is given by the embedding $F_1(X) \hookrightarrow Gr(2,r+1)$.
Then the quotient $\bB(X)/SL(2)$ is isomorphic to $\cM_{0,1}(\PP \cW ,2 )$
where $\cW$ is the universal rank $2$ bundle over $F_1(X)$.
Let
\[
\Theta_0^2(X):=\bB(X) \times_{X} \cM_{0,1}(X ,1 )
\]
be the fiber product which is given by the evaluation maps at the marked points.
Then we have
$$\Gamma_0^2(X)=\Theta_0^2(X)/SL(2).$$
The smoothness of the evaluation map in item (2) of Lemma \ref{assu5-1} implies that
$\Theta_0^2(X)$ is a locally trivial fiber bundle with fiber
$F_1(X,x)$ over a $[ \PP^1 \times \PP(\Sym^2 \CC^2 \otimes \CC^2)]^s$-bundle over $F_1(X)$.
Recall that $F_1(X,x)$ is the moduli space of lines in $X$
which pass through a given point $x$ in $X$.
We remark here that $\Gamma_0^2(X)$ is in fact a disjoint union of
\[\Gamma_0^2(X, \beta)= \coprod_{2\gamma+\delta=\beta}\cM_{0,1}(\PP\cW|_{F_1(X,\gamma)},2)\times_X \cM_{0,1}(X ,\delta)\]
where the fiber product is given by the evaluation maps and $2\gamma+\delta =\beta \in  H_2(X,\ZZ)$ such that $i_*(\gamma)=i_*(\delta)=[\PP^1]\in H_2(\PP^r,\ZZ)$.

Let $q_1: Q_1(X)\lr Q_0(X)$ be the blow-up along $\Theta_0^1(X)$.
Let $\Theta_1^1(X)$ be the exceptional divisor of $q_1$ and let $\Theta_1^2(X)$
be the proper transform of $\Theta_0^2(X)$.
Then we apply elementary modification along the divisor $\Theta_1^1(X)$ to define
$$\cE_{1}(X)=
elm_{\Theta_{1}^{1}(X)}(( 1_X \times q_{1} )^*\cE_{0}(X) , \mathcal{A}_{1} )$$
over $X\times Q_1(X)$.
The destabilizing quotient sheaf $ \mathcal{A}_{1}$ can be described
as follows. Let $y\in \Theta_1^1(X)$.
At $q_1(y)\in Q_0(X)$ which is represented by a stable map $f:C\lr L\subset X$ for some line $L$ in $X$ ,
$\cE_0(X)$ is
\begin{equation}\label{eq-35}
f_* \cO_C = \cO_L \oplus \cO_L (-1)^{\oplus 2}
\end{equation}
and the quotient sheaves
$$\mathcal{A}_{1}|_{X\times \{y\}} = \cO_L (-1)^{\oplus 2}$$
form the flat family $\cA_1$ of destabilizing quotients as in Example \ref{25}.
\begin{prop}\label{pro31}
$\cE_1 (X)|_{X\times\{y\}}$ is a stable sheaf if and only if
$y \in Q_1(X)-[\Theta_1^2 (X) \cup  \Theta_1^3(X)]$,
where $\Theta_1^3(X)$ is a smooth subvariety of $\Theta_1^1(X)$,
which is a $\PP^1 \times \PP^{t-1}$-bundle over $\Theta_0^1(X)$.
\end{prop}
\begin{proof}
This proof is the same as that of Lemma 4.6 in \cite{CK}.
So we only sketch the key ideas.
As mentioned in Proposition \ref{pro1}, elementary modification interchanges the destabilizing subsheaf
and the destabilizing quotient sheaf \cite[Lemma 4.6]{CK}.
In this case, for $y\in\Theta^1_1(X)$, the sheaf $\cE_1 (X)|_{X\times\{y\}}$ fits into a short exact
sequence
$$
	\ses{\cO_L (-1)^{\oplus 2}}{\cE_1 (X)|_{X\times \{y\}}}{\cO_L}.
$$
Moreover by studying deformation theory we obtain isomorphisms
\begin{equation}\label{eq31}
N_{\Theta_0^1(X)/Q_0(X),q_1(z)}\cong \Hom_X (I_L,\cO_L (-1))^{\oplus 2} \cong \Ext_X^1(\cO_L,\cO_L (-1))^{\oplus 2}.
\end{equation}
Furthermore, the extension class $(v,w) \in \Ext_X^1(\cO_L,\cO_L (-1))^{\oplus 2}-\{0\}$ for $\cE_1 (X)|_{X\times \{y\}}$ above determines the line of $y$ in $N_{\Theta_0^1(X)/Q_0(X),q_1(y)}$.
In particular, if $v,w$ are linearly independent, then $\cE_1 (X)|_{X\times \{y\}}$ is stable.

If $v$ and $w$ are linearly dependent, by linear algebra,
$$
\cE_1 (X)|_{X\times\{y\}} \cong F \oplus \cO_L(-1),
$$
where $F$ is a non-split extension of $\cO_L$ by $\cO_L (-1)$.
In particular, $\cE_1 (X)|_{X\times\{y\}} $ is not a stable sheaf. We define such locus as $\Theta_1^3(X)$ in $\Theta_1^1(X)$.
It is obvious that $\Theta_1^3(X)$
is a $\PP^1 \times \PP^{t-1}$-bundle over $\Theta_0^1(X)$.
\end{proof}
\begin{rema}\label{rem4.3}
The dimension $t=\dim\Ext_X^1(\cO_L,\cO_L (-1))= \int_{\frac{\beta}{3}}c_1(T_X) -2$ is constant in each irreducible component $\bM(X,\beta)$
such that $i_*\frac{\beta}{3}=[\PP^1] \in H_2(\PP^r,\ZZ)$ (cf. Remark \ref{rem3.5}).
\end{rema}
The isomorphism in \eqref{eq31} gives us the following.
\begin{coro}\label{trans2}
$Q_0(X)$ intersects with $\Theta_0^1(\PP^r)$ cleanly along $\Theta_0^1(X)$ in $Q_0(\PP^r)$. Hence, $Q_1(X)$ is the proper transform of $Q_0(X)$ via the blow-up $q_1:Q_1(\PP^r)\to Q_0(\PP^r)$.
\end{coro}
The proof is identical to that of Lemma \ref{trans1}.

\begin{rema}
Recall that $\Theta^1_0(X)$ is a $\PP(\Sym^3\CC^2\otimes\CC^2)^s$-bundle over $F_1(X)$.
From the proof of Proposition \ref{pro31}, we find that the isomorphism type of $\cE_1 (X)|_{X\times\{y\}}$ is constant on the fibers $\PP(\Sym^3\CC^2\otimes\CC^2)^s$.
\end{rema}

Let $q_2:Q_2(X)\to Q_1(X)$ be the blow-up along the proper transform $\Theta_1^2(X)$ of $\Theta^2_0(X)$. Recall that $\Theta^2_0(X)$ is a $F_1(X,x)$-bundle over the smooth variety $\bB(X)$ which is a $[\PP^1\times\PP(\Sym^2\CC^2\otimes\CC^2)]^s$-bundle over $F_1(X)$.
There is a tautological section of $\Theta^2_0(X)\to \bB(X)$ which chooses the same line chosen in the base $F_1(X)$.
The intersection of $\Theta_0^1(X)$ and $\Theta_0^2(X)$ is precisely this section which is smooth.
The normal direction of $\Theta^1_0(X)\cap\Theta^2_0(X)$ in $\Theta^1_0(X)$ is the smoothing of a node while the normal direction in $\Theta^2_0(X)$ keeps the node.
In particular, $\Theta^2_0(X)$ and $\Theta^1_0(X)$ intersect cleanly along $\Theta^1_0(X)\cap\Theta^2_0(X)$.
Hence $\Theta_1^2(X)$ is smooth and thus $q_2$ is a smooth blow-up.
Furthermore, notice that $\Theta^1_1(X)\cap \Theta_1^2(X)$ is contained in $\Theta^3_1(X)$.
Therefore $\Theta^3_1(X)\cap \Theta_1^2(X)=\Theta^1_1(X)\cap \Theta_1^2(X)$ and it is smooth.
Let $\Theta_2^2(X)$ denote the exceptional divisor of $q_2$ and let $\Theta_2^j(X)$ for $j=1,3$ be the proper transforms of $\Theta_1^j(X)$.
Then we find that $\Theta^j_2(X)$ are smooth subvarieties of $Q_2(X)$.

We next apply elementary modification to $(1_X\times q_2)^*\cE_1(X)$ along $\Theta_2^2(X)$. For each point $y \in \Theta_2^2(X)$. Let $q_2(y)=y_1$.
If $y_1 \in \Theta_1^2(X)-\Theta_1^1(X)$, then it corresponds to a stable map $f:C\to X$
whose image is the union $C'$ of two distinct lines $L_1$ and
$L_2$. Let $L_2$ be the degree 2 component without loss of generality.
By adjunction, we have a subsheaf $\cO_{C'}$
of $\cE_1(X)|_{X\times\{y\}}=f_*\cO_C$ and a short exact sequence
\begin{equation}\label{eq32}
0\to \cO_{C'}\to f_*\cO_C\to
\cO_{L_2}(-1)\to 0.
\end{equation}
Since $\cO_{C'}$ and $\cO_{L_2}(-1)$ are stable by Lemma \ref{dirstab},
we see that $\cO_{L_2}(-1)$ is the destabilizing quotient.
If $y_1 \in \Theta_1^2(X)\cap \Theta_1^1(X)=\Theta^2_1(X)\cap \Theta^3_1(X)$,
we showed in the proof of Proposition \ref{pro31} that
$\cE_1(X)|_{X\times \{y_1\}} \cong F\oplus \cO_L(-1)$ where
$F$ is a non split extension class in $\Ext_X^1(\cO_L ,\cO_{L}(-1))$
and hence the destabilizing quotient at $y_1$
\begin{equation}\label{3000}
\cA_{2}|_{X\times\{y\}}\cong \cO_{L}(-1).
\end{equation}
Therefore, the destabilizing quotients form a flat family $\cA_2$ over the divisor $\Theta^2_2(X)$. Let
$$\cE_{2}(X)=
elm_{\Theta_{2}^{2}(X)}(( 1_X \times q_{2} )^*\cE_{1}(X) , \mathcal{A}_{2} )$$
over $X\times Q_2(X)$.

\begin{prop}\label{pro32}
$\cE_2(X)|_{X\times\{y\}}$ is a stable sheaf for $y\in Q_2(X)- \Theta_2^3 (X)$.
\end{prop}
\begin{proof}
The proof of this proposition is identical to that of Lemma 4.10 of \cite{CK}. So we omit it.
\end{proof}
The same argument in the proof of Lemma \ref{trans1} also proves the following.
\begin{coro}\label{trans3}
$Q_1(X)$ intersects with $\Theta^2_1(\PP^r)$ cleanly along $\Theta^2_1(X)$ in $Q_1(\PP^r)$. Hence $Q_2(X)$ is the proper transform of $Q_1(X)$ via the blow-up $q_2:Q_2(\PP^r)\to Q_1(\PP^r)$.
\end{coro}

Let $q_3:Q_3(X)\to Q_2(X)$ be the blow-up along the smooth subvariety $\Theta_2^3(X)$.
Let $\Theta^3_3(X)$ denote the exceptional divisor and $\Theta^j_3(X)$ be the proper transforms of $\Theta^j_2(X)$ for $j=1,2$.
From our analysis of $\cE_2(X)$ above, we find that for $y\in \Theta^3_3(X)$, $\cO_L(-1)$ is the destabilizing quotient for some line $L$ in $X$.
Hence these form a flat family $\cA_3$ of quotients. We let
$$\cE_{3}(X)=
elm_{\Theta_{3}^{3}(X)}(( 1_X \times q_{3} )^*\cE_{2}(X) , \mathcal{A}_{3} )$$
over $X\times Q_3(X)$.

For $i,j=1,2,3$, we define $$\bM_i(X)=Q_i(X)/SL(2),\qquad \Gamma_i^j(X)=\Theta_i^j(X)/SL(2).$$
Since $Q_0(X)$ is the stable part of a smooth projective variety which has no strictly semistable points, $\bM_i(X)$ are projective and the induced morphisms
\[ \bM_3(X)\lra \bM_2(X)\lra \bM_1(X)\lra \bM(X)\]
are (weighted) blow-ups.
By the same proof as \cite[Lemma 4.13]{CK}, we obtain the following.
\begin{prop}\label{108211}
$\cE_3(X)$ is a family of stable sheaves on $X$ parameterized by $Q_3(X)$.
Therefore there is an $SL(2)$-invariant morphism $\psi_X:Q_3(X)\to \bS(X)$ which induces a birational morphism $\bar\psi_X:\bM_3(X)\to \bS(X)$.
\end{prop}

In the remaining part of this subsection we show that
the morphism $\bar{\psi}_X$
can be factorized into a sequence of weighted blow-ups.
To do this, we analyze analytic neighborhoods of
the exceptional divisors $\Gamma_3^i (X), i=1,2,3$ in $\bM_3(X)$.
First the analytic neighborhood of $\Gamma_3^1 (X)$
is very similar to the analytic neighborhood of $\Gamma_3^1 (\PP^r)$
because of the following two lemmas.
\begin{lemm}\label{norline}
The normal bundle of a line $L\cong \PP^1$ in $X$ is
$$N_{L /X} \cong \cO_{\PP^1}(1)^{\oplus k}\oplus \cO_{\PP^1}^{\oplus l}$$
for some integers $k,l$ with $k+l=\dim X-1$.
\end{lemm}
\begin{proof}
Since $N_{L/\PP^r}=\cO_{\PP^1}(1)^{\oplus r-1}$, $N_{L/X}$ is a subbundle of $\cO(1)^{\oplus r-1}$. Since $X$ is convex, we cannot have negative factors and hence the lemma follows.
\end{proof}

\begin{lemm}\label{norbun}
The normal bundle to $\Theta^1_0(X)$ in $Q_0(X)$ restricted to
a fiber $\PP(\Sym^3(\CC^2)\otimes \CC^2)^s=(\PP^7)^s$ is
$$\cO_{(\PP^7)^s}(-1)^{\oplus 2t}$$ where $t=\dim F_1(X,x)=\dim \Ext^1_X(\cO_L,\cO_L(-1))$.
\end{lemm}
\begin{proof}
The tangent space at $L$ in $F_1(X)$ has dimension $\dim H^0(L,N_{L/X})=2k+l$ by Lemma \ref{norline} and thus the pull-back of the tangent bundle of $F_1(X)$ to a fiber $\PP(\Sym^3(\CC^2)\otimes \CC^2)^s$ is $\cO^{\oplus 2k+l}$. The dimension of $F_1(X)$ is thus $2k+l=k+\dim X-1$ which must be equal to $t+\dim X-1$ from the smooth fibrations $F_1(X)\leftarrow \cM_{0,1}(X,1)\rightarrow X$ with fibers $\PP^1$ and $F_1(X,x)$ respectively. Therefore $t=k$.

The rest of the proof follows directly from \cite[Lemma 4.2]{CK}.
\end{proof}
From Lemma \ref{norbun}, an analytic neighborhood $\mathcal{U}^1(X)$
of $\Gamma_0^1(X)$ in $\bM(X)$ is a bundle over $F_1(X)$
with fiber
$$
\widetilde{\mathcal{U}}^1(X)=\cO_{\PP^7}(-1)^{\oplus 2t}/\!/SL(2).
$$
Note that $t$ may be different in different components $\bM(X,\beta)$
such that $i_*\frac{\beta}{3}=[\PP^1] \in H_2(\PP^r)$ (cf. Remark \ref{rem4.3}).
By blowing up, an analytic neighborhood of $\Gamma_1^1(X)$
in $\bM_1(X)$ is a bundle over $F_1(X)$ with fiber
\begin{equation}\label{eq34}
\widetilde{\mathcal{U}}_1^1(X)
:=\cO_{\PP^7 \times \PP^{2t-1}}(-1,-1)/\!/_{\cO(1, \lambda )}SL(2),\quad 0<\lambda <\!<1
\end{equation}
Now let $\lambda$ in \eqref{eq34} vary from $0^+$ to $\infty$.
This variation has been worked out in \cite[(4.19)]{CK}:
The GIT quotient undergoes two blow-ups and two blow-downs and
the two blow-ups coincide with the quotients of $q_2$ and $q_3$ by $SL(2)$.
Therefore we can blow down the inverse image $\cU^1_3(X)$ in $\bM_3(X)$ of $\cU^1(X)$ three times.
Likewise we can analyze a neighborhood of $\Gamma^2_0(X)$ to conclude that $\bM_3(X)$ can be blown-down three times
\[
\bM_3(X)\mapright{q_4} \bM_4(X)\mapright{q_5}\bM_5(X)\mapright{q_6}\bM_6(X).
\]
Then we can check that the morphism $\bar{\psi}_X:\bM_3(X)\to \bS(X)$ factors through a morphism $\bM_6(X)\to \bS(X)$ which is bijective. This is enough to conclude that $\bM_6(X)\cong\bS(X)$ and $\bS(X)$ is the consequence of three blow-downs from $\bM_3(X)$.
The details are exactly the same as the proof in \cite[\S4.4]{CK}.
In summary, we obtain the following.

\begin{theo}\label{thm2} Let $X$ be a projective homogeneous variety. Then
$\bS(X)=\bS(X,3)$ is obtained from $\bM(X)=\bM(X,3)$ by blowing up
along $\Gamma_0^1(X),\Gamma_1^2(X),\Gamma_2^3(X)$
and the blowing down along $\Gamma_3^2(X),\Gamma_4^3(X),\Gamma_5^1(X)$
where $\Gamma^j_i(X)$ is the
proper transform of $\Gamma^j_{i-1}(X)$ if $\Gamma^j_{i-1}(X)$ is not
the blow-up/-down center and the image/preimage of
$\Gamma^j_{i-1}(X)$ otherwise.
\end{theo}
\begin{rema}
To prove Theorem \ref{thm2}, we only used the items (1),(2) in Lemma \ref{assu5-1}.
Thus we can say that Theorem \ref{thm2} holds for any smooth projective variety $X$
satisfying conditions (1),(2) in Lemma \ref{assu5-1}.
\end{rema}


\subsection{Comparison of $\bS(X)$ and $\bH(X)$}\label{sec3.2}
The goal of this subsection is to show
that $\bH(X)$ is a smooth blow-up of $\bS(X)$
along the locus of the planar stable sheaves.
In \cite{CK}, when $X=\PP^r$, we have applied the Fujiki-Nakano criterion \cite{Fujiki}
for showing that the divisorial contraction $\bH(\PP^r)\lr \bS(\PP^r)$ is a smooth blow-down.
But for general $X$, it seems difficult to check that
$\bH(X)$ is smooth, and so we use the results of the previous subsection and Proposition 3.3 in \cite{CK} instead.
In this subsection only, we use all properties in Lemma \ref{assu5-1} for $X$.

Note that $\cU^1_i(X):=q_i(\cU^1_{i-1}(X)), (i=4,5,6)$ is a bundle over $F_1(X)$ by the construction of $q_i$.
\begin{prop}\label{306}
 $\bS(X)$ is smooth.
\end{prop}
\begin{proof}
By Theorem \ref{thm2}, $\bS(X)-\bigcup_{j=1}^3 \Gamma_6^j$ is isomorphic to $\bM(X)-\Gamma^1_0(X)\cup\Gamma^2_0(X)$. By Lemma \ref{assu5-1} (1), $\bM(X)-\Gamma^1_0(X)\cup\Gamma^2_0(X)$ is smooth since by definition the automorphism groups are all trivial. For smoothness, near $\Gamma_6^1(X)$ and $\Gamma_6^3(X)$, we look at the variations $\cU^1_i(X)$ in the last part of \S\ref{sec3.1}. Since the fiber $\widetilde{\cU}^1_6(X)=\cO_{\PP^{2t-1}}(-1)^{\oplus 8}\git SL(2)$
of $\cU^1_6(X)$ over $F_1(X)$ is a vector bundle over $Gr(2,t)$ which is smooth, $\cU_6^1(X)$ is smooth as well.
Hence $\bS(X)$ is smooth in a neighborhood of $\Gamma_6^1(X)\cup \Gamma_6^3(X)$.
Similarly from the analysis of neighborhoods $\cU_i^2(\PP^r)$ of $\Gamma_i^2(\PP^r)$ in \cite[\S4.4]{CK}, it is immediate to check that $\cU^2_6(X)$ is smooth. Therefore,
$\bS(X)$ is indeed smooth everywhere.
\end{proof}

Let $\Delta(X)\subset \bS(X)$ be the locus of stable sheaves whose scheme theoretic support is contained in a plane.
When $X=\PP^r$, $\Delta (\PP^r)$ is a
$\bS(\PP^2)$-bundle over $Gr(3,r+1)$
as shown in the proof of \cite[Proposition 3.3]{CK}. More precisely, $$\Delta(\PP^r)=\bS(\PP\cU) \lra Gr(3,r+1)$$
where $\cU$ is the tautological rank 3 vector bundle on $Gr(3,r+1)$ and $\bS(\PP\cU)$ is the relative Simpson moduli space in the obvious sense.
In particular, each $F\in \Delta(\PP^r)$ is contained in a unique plane in $\PP^r$. For $X\subset \PP^r$, it is obvious that set-theoretically
$$\Delta (X)=\bS(X) \cap \Delta (\PP^r).$$

Item (4) of Lemma \ref{assu5-1} implies the following.
\begin{lemm}\label{lem822-1}
If $F\in\Delta(X)\subset \Delta(\PP^r)$, then the unique plane $\Lambda$ containing the support of $F$ is entirely contained in $X$.
\end{lemm}
\begin{proof}
By Lemma \ref{assu5-1} (4), the defining ideal of $X$ in $\PP^r$ is generated by quadratic polynomials. Therefore, the intersection of $X$ with a plane $\Lambda$ is a subvariety of $\PP^r$ whose defining ideal is generated by linear or quadratic polynomials only. Then it is obvious that $X\cap \Lambda$  cannot contain a cubic curve unless $X\cap \Lambda=\Lambda$, i.e. $\Lambda\subset X$. By our choice of Hilbert polynomial $3m+1$, the support $C$ of $F$ is a cubic curve in $X\cap \Lambda$ where $\Lambda$ is the unique plane containing $C$. Therefore, $\Lambda\subset X$ as desired.
\end{proof}

An immediate corollary of Lemma \ref{lem822-1} is the following.
\begin{coro}\label{348}
\begin{enumerate}
  \item $\Delta (X)$ is a $\bS(\PP^2)$-bundle over $F_2(X)$.
  \item $\Delta(\PP^r)$ intersects with $\bS(X)$ cleanly.
\end{enumerate}
\end{coro}
\begin{proof} Let $F_2(X)\subset F_2(\PP^r)=Gr(3,r+1)$ denote the moduli space of all planes in $X$. Then Lemma \ref{lem822-1} gives us a Cartesian diagram
\[\xymatrix{
\Delta(X)\ar@{^(->}[r]\ar[d] &\Delta(\PP^r)\ar[d]\\
F_2(X)\ar@{^(->}[r] & Gr(3,r+1)
}\]
which is exactly (1).

By Lemma \ref{assu5-1} (3), $\Delta(X)$ is smooth.
To show that $T_{\Delta(\PP^r),F}\cap T_{\bS(X),F}= T_{\Delta(X),F}$, it suffices to show that
$$T_{\Delta(X),F}= ker\left( T_{\Delta(\PP^r),F}\hookrightarrow T_{\bS(\PP^r),F}
\rightarrow N_{\bS(X)/\bS(\PP^r),F}\right).$$
Let $v\in T_{\Delta(\PP^r),F}$ be a morphism $$\spec\CC[\epsilon]/(\epsilon^2)\to \Delta(\PP^r) \to Gr(3,r+1).$$
By trivializing the tautological bundle of $Gr(3,r+1)$ over $\spec\CC[\epsilon]/(\epsilon^2)$, we obtain a flat family $\cF$ of sheaves
on $\PP^2\times \spec\CC[\epsilon]/(\epsilon^2)$ and a closed immersion $\mu:\PP^2\times \spec\CC[\epsilon]/(\epsilon^2)\to \PP^r\times \spec\CC[\epsilon]/(\epsilon^2)$ such that $\mu_*\cF$ is the family of stable sheaves on $\PP^r\times \spec\CC[\epsilon]/(\epsilon^2)$ given by $v$. Suppose $v\in ker( T_{\Delta(\PP^r),F}
\rightarrow N_{\bS(X)/\bS(\PP^r),F})$. Then $\mu_*\cF$ has support in $X\times \spec\CC[\epsilon]/(\epsilon^2)$. By the argument of the previous paragraph, the image of $\mu$ has to lie entirely in $X\times \spec\CC[\epsilon]/(\epsilon^2)$. This implies that $v:\spec\CC[\epsilon]/(\epsilon^2)\to \Delta(\PP^r)$ factors through $\Delta(X)$, i.e. $v\in T_{\Delta(X),F}$ as desired.
\end{proof}

Let $\widetilde{\bS}(X)$ be the blow-up of $\bS(X)$ along $\Delta(X)$.
Then by Corollary \ref{348}, $\widetilde{\bS}(X)$ is the proper transform of $\bS(X)$ via the blow-up $\bH(\PP^r)\to \bS(\PP^r)$. On the other hand, $\bH(X)$ is also a proper transform of $\bS(X)$ by its definition as the closure of the locus of smooth curves. Therefore $\bH(X)=\widetilde{\bS}(X)$.

In summary we have the following.
\begin{theo}\label{305}
Let $X$ be a projective homogeneous variety in $\PP^r$. Then
$\bH(X)$ is the smooth blow-up of $\bS(X)$ along the locus $\Delta (X)$ of
planar stable sheaves on $X$.
Moreover there exists a commutative diagram
$$\xymatrix{\bH(X)\ar[d]\ar@{^(->}[r]&\bH(\PP^r)\ar[d]\\
\bS(X)\ar@{^(->}[r]&\bS(\PP^r)}$$
of blow-ups.
\end{theo}
\begin{rema}
Theorem \ref{305} holds for any smooth projective variety satisfying all items in Lemma \ref{assu5-1}.
\end{rema}
In section \S4, we will use Theorems \ref{thm2} and \ref{305}
to calculate the Poincar\'e polynomials of $\bS(X)$ and $\bH(X)$ when $X=Gr(k,n)$.


\section{Calculation of the Poincar\'e polynomials }\label{sec4}
For a variety $Z$, let
\[ P(Z)=\sum_{i\ge 0} \dim_{\QQ} H^{i}(Z,\QQ) q^{i/2} \]
be the Poincar\'e polynomial of $Z$. For every variety below, the odd degree cohomology will be trivial and thus $P(Z)$ will be a polynomial.
In this section, we calculate the Poincar\'e polynomials of $\bS(X,d)$ and $\bH(X,d)$ for $d=2,3$ when $X$ is the Grassmannian $Gr(k,n)$ of $k$ dimensional subspaces in $\CC^n$ with $k< n$.

We begin with the following lemma.
\begin{lemm}\label{polem}
\begin{enumerate}
\item $P(F_1(Gr(k,n)))=\prod_{i=1}^{k+1} \frac{1-q^{n-i+1}}{1-q^i}\cdot \prod_{i=1}^{k-1}
\frac{1-q^{k-i+2}}{1-q^i}.$
  \item $F_2(Gr(k,n))$
  is the disjoint union of two nonsingular varieties;
 a $Gr(k-2,k+1)$-bundle over $Gr(k+1,n)$ and a $Gr(k-1,k+2)$-bundle over $Gr(k+2,n)$.
  \item Let $\mathrm{ev}:\cM_{0,1}(Gr(k,n),1) \rightarrow Gr(k,n) $ be the
  evaluation map at the marked point so that $\mathrm{ev}^{-1}(x)=F_1(Gr(k,n),x)$. Then $P(\mathrm{ev}^{-1}(x))=\frac{(1-q^{n-k})(1-q^k)}{(1-q)^2}$.
\end{enumerate}
\end{lemm}
\begin{proof}
(1) A line in $Gr(k,n)$ is the space of all $k$ dimensional subspaces
which is contained in a fixed $k+1$ dimensional subspace in $\CC^n$
and contains a fixed $k-1$ dimensional subspace by \cite[Exercise 6.9]{Harris}.
Therefore $F_1(Gr(k,n))$ is
a $Gr(k-1,k+1)$-bundle over $Gr(k+1,n)$.
Hence (1) follows from the well-known formula
$$
P(Gr(k,n))= \prod_{i=1}^k \frac{1-q^{n-i+1}}{1-q^i}.
$$

(2) See \cite[Theorem 4.9]{Landsberg}.

(3) Let $X=Gr(k,n)$. The two fibrations
$$
\xymatrix{\cM_{0,1}(X,1) \ar[r] \ar[d] & X \\
F_1(X)&}
$$
give rise to
$$P(\mathrm{ev}^{-1}(x))\cdot P(Gr(k,n)) = P(\PP^1)\cdot P(Gr(k-1,k+1)) \cdot P(Gr(k+1,n)).$$
Therefore (3) follows from
$$P(\mathrm{ev}^{-1}(x))= \frac{P(\PP^1)\cdot P(Gr(k-1,k+1)) \cdot P(Gr(k+1,n))}{P(Gr(k,n))}.$$
\end{proof}

The Poincar\'e polynomials of $\bM(Gr(k,n),d)$ for $d=2,3$ were calculated by A. L\'opez-Mart\'in as follows.
\begin{theo} \cite{Martin}\label{thm40}
(1) The Poincar\'e polynomial of $\bM(Gr(k,n),2)$ is
$$\frac{((1+q^n)(1+q^3)-q(1+q)(q^k+q^{n-k}))\prod_{i=k}^{n}
(1-q^i)}{(1-q)^2(1-q^2)^2 \prod_{i=1}^{n-k-1} (1-q^i)}.$$
(2) The Poincar\'e polynomial of $\bM(Gr(k,n),3)$ is
\small$$\frac{F_1(q) (1+q^{2n})+(1+q)^2(F_2(q) q^n (1+q^2)-F_3(q) q (1+q^n )(q^k +q^{n-k}))+F_4(q) q^2 (q^{2k} +q^{2n-2k})}{(1-q)(1-q^2)^2 (1-q^3)^2}$$
$$\cdot  P(Gr(k+1,n))\cdot P(Gr(k-1,k+1))$$\normalsize
where \[ F_1(q)=1+2q^2+3q^3+3q^4-q^5+q^6-3q^7-3q^8-2q^9-q^{11}, \]
\[ F_2(q)=1+5q^2+2q^3-2q^4-5q^5-q^7, \]
\[ F_3(q)=2+3q^2+q^3-q^4-3q^5-2q^7,\]
\[ F_4(q)=1+6q+3q^2+2q^3-2q^4-3q^5-6q^6-q^7.  \]
\end{theo}

\subsection{d=2 case}
By the comparison result of \S2 and the blow-up formula (\cite[p.605]{GriffithsHarris}), we obtain the following.
\begin{coro}\label{thm41}
\small
$$P(\bS(Gr(k,n),2))=\frac{[(1+q^n)(1+q^3)-q(1+q)(q^k+q^{n-k})+(1-q^2)(q^3-q^{n-2})]\prod_{i=n-k}^{n} (1-q^i)}{(1-q)^2(1-q^2)^2\prod_{i=1}^{k-1} (1-q^i)},$$
\normalsize
where $\prod_{i=1}^{0} (1-q^i)$ is defined to be $1$.
\end{coro}
\begin{proof} By Theorem \ref{thm1}, the blow-up of $\bM(Gr(k,n),2)$ along a $\bM(\PP^1,2)$-bundle over $F_1(Gr(k,n))$ coincides with the blow-up of $\bS(Gr(k,n),2)$ along a $\PP(\Ext^1(\cO_L,\cO_L(-1)))$-bundle over $F_1(Gr(k,n))$.
By Lemma \ref{polem} and Theorem \ref{thm40}, we obtain
$$P(\bS(Gr(k,n),2))=P(\bM(Gr(k,n),2))+P(Gr(k-1,k+1))P(Gr(k+1,n))P(\PP^2)(P(\PP^{n-3})-1)$$
$$-P(Gr(k-1,k+1))P(Gr(k+1,n))P(\PP^{n-3})(P(\PP^{2})-1)$$
$$=P(\bM(Gr(k,n),2))+P(Gr(k-1,k+1))P(Gr(k+1,n))(P(\PP^{n-3})-P(\PP^{2}))$$
$$=\frac{((1+q^n)(1+q^3)-q(1+q)(q^k+q^{n-k}))\prod_{i=k}^{n} (1-q^i)}{(1-q)^2(1-q^2)^2 \prod_{i=1}^{n-k-1} (1-q^i)}$$
$$+\frac{(1-q^{k+1})}{(1-q)}\cdot \frac{(1-q^k)}{(1-q^2)}\cdot \frac{\prod_{i=n-k}^{n} (1-q^i)}{\prod_{i=1}^{k+1} (1-q^i)}\cdot(\frac{1-q^{n-2}}{1-q}-\frac{1-q^3}{1-q})$$
\small
$$=\frac{\prod_{i=n-k}^{n} (1-q^i)}{\prod_{i=1}^{k-1} (1-q^i)}\cdot \frac{1}{(1-q)^2 (1-q^2)^2} \cdot ((1+q^n)(1+q^3)-q(1+q)(q^k+q^{n-k})+(1-q^2)(q^3-q^{n-2})).$$
\normalsize
\end{proof}

\subsection{d=3 case}
Theorems \ref{thm2} and \ref{305} enable us to calculate the Poincar\'e polynomials of $\bS(Gr(k,n),3)$ and $\bH(Gr(k,n),3)$ as follows.
\begin{coro}\label{thm42}
\begin{enumerate}
  \item The Poincar\'e polynomial of $\bS(Gr(k,n),3)$ is
\small
$$\{ \frac{F_1(q) (1+q^{2n})+(1+q)^2(F_2(q) q^n (1+q^2)-F_3(q) q (1+q^n )(q^k +q^{n-k}))+F_4(q) q^2 (q^{2k} +q^{2n-2k})}{(1-q)(1-q^2)^2 (1-q^3)^2}$$
$$ +  (1+q+2q^2 +q^3 + q^4) (\frac{1-q^{2n-4}}{1-q}-1) $$
$$+ \frac{1-q^{2}}{1-q}(\frac{(1-q^{n-k})(1-q^k)}{(1-q)^2}+ \frac{1-q^{n-2}}{1-q}-1 ) (1+q+q^2)(\frac{1-q^{n-1}}{1-q}-1)$$
$$ + \frac{1-q^{n-2}}{1-q} ((1+q)(1+q+2q^2 +q^3 + q^4)+q (1+q)(1+q+q^2) ) (\frac{1-q^{n-2}}{1-q}-1) $$
$$-\frac{1-q^{2}}{1-q} [ \frac{1-q^{n-1}}{1-q}(\frac{(1-q^{n-k})(1-q^k)}{(1-q)^2}+ \frac{1-q^{n-2}}{1-q}-1 )+\frac{1-q^{2}}{1-q} \frac{1-q^{n-2}}{1-q}(\frac{1-q^{n-2}}{1-q}-1)](\frac{1-q^{3}}{1-q} -1)$$
$$ - \frac{1-q^{2}}{1-q} \frac{1-q^{n-2}}{1-q}\frac{1-q^{n-2}}{1-q}(\frac{1-q^{5}}{1-q} -1)  $$
$$ - \frac{1-q^{n-2}}{1-q}\frac{1-q^{n-3}}{1-q^2}(\frac{1-q^{8}}{1-q}- 1)\}\cdot
\prod_{i=1}^{k+1} \frac{1-q^{n-i+1}}{1-q^i}\cdot \prod_{i=1}^{k-1}
\frac{1-q^{k-i+2}}{1-q^i}.$$
\normalsize
\item The Poincar\'e polynomial of $\bH(Gr(k,n),3)$ is
$$P(\bS(Gr(k,n),3))$$
\small
$$+(1+2q+3q^2+3q^3+3q^4+3q^5+3q^6+2q^7+q^8)\cdot \{\prod_{i=1}^{k+1} \frac{1-q^{n-i+1}}{1-q^i} \cdot \prod_{i=1}^{k-2} \frac{1-q^{k-i+2}}{1-q^i}\cdot(\frac{1-q^{2n-k-4}}{1-q}-1)$$
$$+ \prod_{i=1}^{k+2} \frac{1-q^{n-i+1}}{1-q^i} \cdot\prod_{i=1}^{k-1} \frac{1-q^{k-i+3}}{1-q^i}\cdot(\frac{1-q^{n+k-4}}{1-q}-1)\}.$$
\normalsize
\end{enumerate}
\end{coro}

\begin{proof}
(1) By Theorem \ref{thm2} and the blow-up formula of cohomology groups, we have
$$P(\bS(Gr(k,n),3))=P(\bM(Gr(k,n),3)) + P(F_1(X)) P(\bM(\PP^1,3)) (P(\PP^{2n-5})-1) $$
$$+ P(X) P( bl_{\Delta } (\mathrm{ev}^{-1}(x) \times \mathrm{ev}^{-1}(x)) ) P(\bM(\PP^1,2)) (P(\PP^{n-2})-1)$$
$$ +P(F_1(X)) P(\PP^{n-3}) ((1+q)(1+q+2q^2 +q^3 + q^4)+q (1+q)(1+q+q^2) ) (P(\PP^{n-3})-1) $$
$$ -  P(X) [ P( bl_{\Delta } (\mathrm{ev}^{-1}(x) \times \mathrm{ev}^{-1}(x)) ) P(\PP^{n-2})+P(\PP^1)P(\mathrm{ev}^{-1}(x))P(\PP^{n-3})(P(\PP^{n-3})-1)](P(\PP^2_{(1,2,2)}) -1) $$
$$ - P(F_1(X)) P(\PP^1 \times \PP^{n-3}) P(\PP^{n-3})(P(\PP^4_{(1,2,2,3,3)}) -1)  $$
$$ - P(F_1(X)) P(Gr(2,n-2))(P(\PP^7 )- 1).$$
Then (1) immediately follows from this.

(2) By Theorem \ref{305} and Lemma \ref{polem},
$$P(\bH(Gr(k,n),3))=P(\bS(Gr(k,n),3))$$
$$+ P(Gr(k+1,n))\cdot P(Gr(k-2,k+1))\cdot P(\bS(Gr(1,3),3))\cdot(\frac{1-q^{2n-k-4}}{1-q}-1)$$
$$+P(Gr(k+2,n))\cdot P(Gr(k-1,k+2))\cdot P(\bS(Gr(1,3),3))\cdot(\frac{1-q^{n+k-4}}{1-q}-1).$$
If we use (1) for $P(\bS(Gr(1,3),3))$, we obtain (2).
\end{proof}

\bibliographystyle{amsplain}

\end{document}